\documentclass[11pt]{amsart}

\usepackage{amsmath,amssymb}
\usepackage{calrsfs}
\usepackage{url}
\newcommand{\Z}{{\mathbb{Z}}}
\newcommand{\C}{{\mathbb{C}}}
\newcommand{\F}{{\mathbb{F}}}

\newcommand{\fg}{{\mathfrak{g}}}
\newcommand{\fp}{{\mathfrak{p}}}
\newcommand{\ft}{{\mathfrak{t}}}
\newcommand{\fz}{{\mathfrak{z}}}
\newcommand{\cG}{{\mathcal{G}}}
\newcommand{\cN}{{\mathcal{N}}}
\newcommand{\cU}{{\mathcal{U}}}
\newcommand{\cC}{{\mathfrak{C}}}
\newcommand{\cO}{{\mathcal{O}}}
\newcommand{\cb}{{\mathbf{b}}}

\newcommand{\hgt}{{\operatorname{ht}}}
\newcommand{\AV}{{\operatorname{AV}}}
\newcommand{\Ind}{{\operatorname{Ind}}}
\newcommand{\Irr}{{\operatorname{Irr}}}
\newcommand{\GL}{{\operatorname{GL}}}
\newcommand{\SL}{{\operatorname{SL}}}

\newcommand{\Gder}{G_{\operatorname{der}}}
\renewcommand{\leq}{\leqslant}
\renewcommand{\geq}{\geqslant}

\newtheorem{thm}{Theorem}[section]
\newtheorem{cor}[thm]{Corollary}
\newtheorem{lem}[thm]{Lemma}
\newtheorem{prop}[thm]{Proposition}
\newtheorem{conj}[thm]{Conjecture}
\theoremstyle{definition}
\newtheorem{defn}[thm]{Definition}
\newtheorem{exmp}[thm]{Example}
\newtheorem{abs}[thm]{}
\theoremstyle{remark}
\newtheorem{rem}[thm]{Remark}
\numberwithin{equation}{section}

\begin{document}

\title[Generalised Gelfand--Graev representations]{Generalised 
Gelfand--Graev representations in bad characteristic~?}
\author{Meinolf Geck}
\address{IAZ - Lehrstuhl f\"ur Algebra\\Universit\"at Stuttgart\\ 
Pfaffenwaldring 57\\D--70569 Stuttgart\\ Germany}
\email{meinolf.geck@mathematik.uni-stuttgart.de}

\date{November 1, 2018}

\begin{abstract} Let $G$ be a connected reductive algebraic group
defined over a finite field with $q$ elements. In the 1980's, Kawanaka
introduced generalised Gelfand-Graev representations of the finite group
$G(\F_q)$, assuming that $q$ is a power of a good prime for $G$. These 
representations have turned out to be extremely useful in various contexts.
Here we investigate to what extent Kawanaka's construction can be carried 
out when we drop the assumptions on~$q$. As a curious by-product, we
obtain a new, conjectural characterisation of Lusztig's concept of 
special unipotent classes of $G$ in terms of weighted Dynkin diagrams.
\end{abstract}

\maketitle

\section{Introduction} \label{sec0}

Let $p$ be a prime and $k=\overline{\F}_p$ be an algebraic closure of the 
field with $p$ elements. Let $G$ be a connected reductive algebraic 
group over $k$ and assume that $G$ is defined over the finite subfield 
$\F_q\subseteq k$, where $q$ is a power of~$p$. Let $F\colon G\rightarrow G$ 
be the corresponding Frobenius map. We are interested in studying the 
representations (over an algebraically closed field of characteristic~$0$) 
of the finite group $G^F=\{g\in G \mid F(g)=g\}$. 

Assuming that $p$ is a ``good'' prime for $G$, Kawanaka \cite{kaw0}, 
\cite{kaw1}, \cite{kaw2} described a procedure by which one can associate 
with any unipotent element $u\in G^F$ a representation $\Gamma_u$ of $G^F$, 
obtained by induction of a certain one dimensional representation from a 
unipotent subgroup of $G^F$. If $u$ is the identity element, then 
$\Gamma_1$ is the regular representation of $G^F$; if $u$ is a regular 
unipotent element, then $\Gamma_u$ is a Gelfand--Graev representation as 
defined, for example, in \cite[\S 8.1]{C2} or \cite[\S 14]{St}. For 
arbitrary $u$, the representation $\Gamma_u$ is called a {\it generalised 
Gelfand--Graev representation} (GGGR for short); it only depends on the
$G^F$-conjugacy class of~$u$. 

A fundamental step in understanding the GGGRs is achieved by Lusztig 
\cite{L2} where the characters of GGGRs are expressed in terms of 
characteristic functions of intersection cohomology complexes on~$G$. In 
\cite{L2} it is assumed that $p$ is sufficiently large; in \cite{tay} 
it is shown that one can reduce these assumptions so that everything works
as in Kawanaka's original approach. These results have several consequences. 
By \cite{gehe} the characters of the various $\Gamma_u$ span the 
$\Z$-module of all unipotently supported virtual characters of~$G^F$. 
In addition to the original applications in \cite{kaw0}, \cite{kaw1}, 
\cite{kaw2}, GGGRs have turned out to be very useful in various
questions concerning $\ell$-modular representations of $G^F$ where $\ell$ 
is a prime not equal to~$p$; see, e.g., \cite{gehi2}, \cite{duma}. Thus, 
it seems desirable to explore the possibilities for a definition without 
any restriction on~$p,q$. These notes arose from an attempt to give such 
a definition. Recall that $p$ is ``good'' for $G$ if $p$ is good for each 
simple factor involved in~$G$; the conditions for the various simple types 
are as follows. 
\begin{center}
$\begin{array}{rl} A_n: & \mbox{no condition}, \\
B_n, C_n, D_n: & p \neq 2, \\
G_2, F_4, E_6, E_7: &  p \neq 2,3, \\
E_8: & p \neq 2,3,5.  \end{array}$
\end{center}
Easy examples indicate that one can not expect a good definition of 
GGGRs for all unipotent elements of $G^F$. Instead, it seems reasonable
to restrict oneself to those unipotent classes which ``come from 
characteristic~$0$'', where the classical Dynkin--Kostant theory is
available; see Section~\ref{sec1}. This is also consistent with the picture 
presented by Lusztig \cite{L3a}, \cite{L3b}, \cite{L3c}, \cite{L3d} for 
dealing with unipotent classes in small characteristic. Based on this
framework, we formulate in Definition~\ref{def24} some precise conditions 
under which it should be possible to define GGGRs for a given unipotent 
class. Of course, these conditions will be satisfied if $p$ is a good 
prime for~$G$, and lead to Kawanaka's original GGGRs; so then the question
is how far we can go beyond this. Our answer to this question is as follows. 

An essential feature of GGGRs is that they are very closely related to the 
``unipotent supports'' of the irreducible representations of $G^F$, in the 
sense of Lusztig \cite{L2}. (In a somewhat different way, and without
complete proofs, this concept appeared under the name of ``wave front 
set'' in Kawanaka \cite{kaw2}.) Let $\cC^\bullet$ be the set of unipotent 
classes of $G$ which arise as the unipotent support of some irreducible 
representation of $G^F$, or of $G^{F^n}$ for some $n\geq 1$. (Thus, since 
$G=\bigcup_{n\geq 1} G^{F^n}$, the set $\cC^\bullet$ only depends on $G$ 
but not on the particular Frobenius map~$F$.) If $p$ is a good prime 
for $G$, then it is known that $\cC^\bullet$ is the set of all 
unipotent classes of $G$. In general, all classes in $\cC^\bullet$ indeed 
``come from characteristic~$0$''. Based on the methods in \cite{L4}, an 
explicit description of the sets $\cC^\bullet$, for $G$ simple and 
$p$ bad, is given in Proposition~\ref{uni2}. This result complements
the general results on ``unipotent support'' in \cite{gema}, \cite{L2} 
and may be of independent interest. 

Now, extensive experimentation (with computer programs written in {\sf GAP} 
\cite{gap4}) lead us to the expectation, formulated as 
Conjecture~\ref{main}, that our conditions in 
Definition~\ref{def24} will work for all the classes in $\cC^\bullet$, 
without any restriction on $p,q$. In Section~\ref{sec4}, we work out in
detail the example where $G$ is of type $F_4$. Our investigations also 
suggest a new characterisation of Lusztig's special unipotent classes;
see Conjecture~\ref{main2} and Corollary~\ref{corf4a}. 

These notes merely contain examples and conjectures; nevertheless we hope 
that they show that the story about GGGRs is by no means complete and 
that there is some evidence for a further theory in bad characteristic.

\section{Weighted Dynkin diagrams} \label{sec1}

We use Carter \cite{C2} as a general reference for results on 
algebraic groups and unipotent classes. 
Let $G,k,p,\ldots$ be as in Section~\ref{sec0}. Also recall that $G$ is 
defined over the finite field $\F_q\subseteq k$, with corresponding 
Frobenius map $F\colon G\rightarrow G$. We fix an $F$-stable maximal torus 
$T \subseteq G$ and an $F$-stable Borel subgroup $B\subseteq G$
containing~$T$. We have $B=U\rtimes T$ where $U$ is the unipotent radical 
of~$B$. Let $\Phi$ be the set of roots of $G$ with respect to~$T$ and 
$\Pi\subseteq\Phi$ be the set of simple roots determined by $B$. Let 
$\Phi^+$ and $\Phi^-$ be the corresponding sets of positive and negative
roots, respectively. Let $\fg=\mbox{Lie}(G)$ be the Lie algebra of $G$. 
Then $G$ acts on $\fg$ via the adjoint representation $\mbox{Ad}\colon 
G\rightarrow \GL(\fg)$. 

\begin{abs} \label{abs11} For each $\alpha\in\Phi$, we have a corresponding 
homomorphism of algebraic groups $x_\alpha\colon k^+\rightarrow G$, $u
\mapsto x_\alpha(u)$, which is an isomorphism onto its image; furthermore, 
$tx_\alpha(u)t^{-1}=x_\alpha(\alpha(t)u)$ for all $t\in T$ and $u\in k$.
Setting $U_\alpha:=\{x_\alpha(u)\mid u\in k\}$, we have $G=\langle T,U_\alpha
\;(\alpha\in \Phi)\rangle$. Note that $U_\alpha\subseteq \Gder$ for all
$\alpha\in\Phi$, where $\Gder$ denotes the derived subgroup of~$G$. On the 
level of $\fg$, we have a direct sum decomposition
\[ \fg=\ft\oplus \bigoplus_{\alpha\in\Phi}\fg_\alpha\]
where $\ft=\mbox{Lie}(T)$ is the Lie algebra of $T$ and $\fg_\alpha$
is the image of the differential $d_0x_\alpha\colon k\rightarrow \fg$;
furthermore, $\mbox{Ad}(t)(y)=\alpha(t)y$ for $t\in T$ and $y\in
\fg_\alpha$. We set 
\[e_\alpha:=d_0x_\alpha(1)\in \fg_\alpha.\]
Then $e_\alpha\neq 0$ and $\fg_\alpha=ke_\alpha$. (For all this see, e.g., 
\cite[\S 8.1]{spr}.) 
\end{abs}

\begin{abs} \label{abs12} For $\alpha\in\Phi$, we can write uniquely 
$\alpha= \sum_{\beta\in \Pi} n_\beta\beta$ where $n_\beta\in\Z$ for 
all $\beta\in\Pi$. Then $\hgt(\alpha):=\sum_{\beta\in\Pi} n_\beta$ is 
called the height of~$\alpha$. We fix once and for all a total ordering 
$\preceq$ of $\Phi^+$ which is compatible with the height, that is, if
$\alpha,\beta \in\Phi^+$ are such that $\alpha\preceq \beta$, then 
$\hgt(\alpha)\leq \hgt(\beta)$. Then every $u\in U$ has a unique expression 
$u=\prod_{\alpha\in\Phi^+} x_\alpha(u_\alpha)$ where $u_\alpha \in k$ 
(and the product is taken in the given order $\preceq$ of $\Phi^+$). Let
$\alpha,\beta\in\Phi^+$, $\alpha\neq \beta$. Let $u,v\in k$. Then we have 
Chevalley's commutator relations 
\[ x_\alpha(u)x_\beta(v)x_\alpha(u)^{-1}x_\beta(v)^{-1}=\prod_{i,j>0;\; 
i\alpha+j\beta\in\Phi} x_{i\alpha+j\beta}(C_{\alpha,\beta,i,j}
u^iv^j)\]
where the constants $C_{\alpha,\beta,i,j}\in k$ only depend on $\alpha,
\beta,i,j$ but not on~$u,v$ (and, again, the product on the right hand side 
is taken in the given order $\preceq$ of $\Phi^+$). Furthermore, if $\alpha+
\beta \in \Phi$, then 
\[[e_\alpha,e_\beta]=N_{\alpha,\beta}e_{\alpha+\beta}\qquad
\mbox{where} \qquad N_{\alpha,\beta}:=C_{\alpha,\beta, 1,1}\in k.\]
(Since all $U_\alpha$, $\alpha\in\Phi$, are contained in the semisimple
algebraic group $\Gder$, this follows from \cite[Lemma~15 (p.~22) and
Remark (p.~64)]{St}.)
\end{abs}

\begin{abs} \label{abs13} It will also be convenient to fix some notation
concerning the action of the Frobenius map $F\colon G\rightarrow G$. 
By the results in \cite[\S 10]{St}, there exist a permutation $\tau \colon 
\Phi\rightarrow\Phi$ and signs $\epsilon_\alpha=\pm 1$ ($\alpha\in\Phi$) 
such that
\[ F(x_\alpha(u))=x_{\tau(\alpha)}(\epsilon_\alpha u^q)\qquad
\mbox{for all $u\in k$}.\]
Here, we can assume that $\tau(\Pi)=\Pi$ and $\epsilon_{\pm \beta}=1$
for all $\beta\in\Pi$.  Now $F$ also induces a Frobenius map on the Lie 
algebra $\fg$ which we denote by the same symbol. We have
$F(uy)=u^qF(y)$ for all $u\in k$ and $y\in\fg$; furthermore, 
\[ F(e_\alpha)=\epsilon_\alpha e_{\tau(\alpha)} \qquad \mbox{for all 
$\alpha\in\Phi$}.\]
Finally, for $g\in G$ and $y\in \fg$, we have $\mbox{Ad}(F(g))(F(y))=
F(\mbox{Ad}(g)(y))$.
\end{abs}

\begin{abs} \label{abs14} Let $G_0$ be a connected reductive algebraic
group over $\C$ of the same type as $G$; let $\fg_0=\mbox{Lie}(G_0)$ be 
its Lie algebra. Then, by the classical Dynkin--Kostant theory (see, e.g., 
\cite[\S 5.6]{C2}), the nilpotent $\mbox{Ad}(G_0)$-orbits in
$\fg_0$ are parametrized by a certain set $\Delta$ of so-called 
{\it weighted Dynkin diagrams}, i.e., maps $d\colon \Phi\rightarrow \Z$ 
such that
\begin{itemize}
\item[(a)] $d(-\alpha)=-d(\alpha)$ for all $\alpha\in\Phi$ and 
$d(\alpha+\beta)=d(\alpha)+d(\beta)$ for all $\alpha,\beta\in\Phi$ 
such that $\alpha+\beta\in\Phi$;
\item[(b)] $d(\beta)\in\{0,1,2\}$ for every simple root $\beta\in \Pi$.
\end{itemize}
Furthermore, these nilpotent orbits in $\fg_0$ are naturally in bijection 
with the unipotent classes of $G_0$ (see \cite[\S 1.15]{C2}). If $G_0$ is 
a simple algebraic group, then the corresponding set $\Delta$ of weighted 
Dynkin diagrams is explicitly known in all cases; see \cite[\S 13.1]{C2} 
and the references there. (Several examples will be given below.) For 
each $d\in \Delta$, we denote by $\cO_d$ the corresponding nilpotent 
orbit in $\fg_0$ and set 
\begin{center}
$\cb_d:=\frac{1}{2}(\dim G_0-\mbox{rank}(G_0)-\dim \cO_d)$.
\end{center}
This is a very useful invariant for distinguishing nilpotent orbits.
(The number $\cb_d$ is also the dimension of the variety of Borel subgroups
of $G_0$ containing an element in the unipotent class corresponding to 
$\cO_d$; see \cite[\S 1.15, \S 5.10]{C2}.). 
\end{abs}

\begin{abs} \label{abs14a} Let us fix $d\in\Delta$. For $i\in\Z$, we 
set $\Phi_i:=\{\alpha\in\Phi \mid d(\alpha)=i\}$ and define
\[ \fg(i):=\left\{ \begin{array}{rl} \bigoplus_{\alpha\in\Phi_i} \fg_\alpha
&\quad \mbox{if $i\neq 0$},\\
\ft\oplus \bigoplus_{\alpha\in\Phi_0} \fg_\alpha &\quad \mbox{if $i=0$}.
\end{array}\right.\]
Thus, as in \cite[\S 2.1]{kaw1}, we obtain a grading $\fg=
\bigoplus_{i\in\Z} \fg(i)$; note that we do have $[\fg(i),\fg(j)] 
\subseteq \fg(i+j)$ for all $i,j\in\Z$. For any $i\geq 0$, we also set
$\fg(\geq i):=\bigoplus_{j\geq i} \fg(j)$. Furthermore, we define subgroups 
of $G$ as follows. 
\begin{align*}
P&:=\langle T,U_\alpha\mid\alpha\in\Phi_i\mbox{ for all $i\geq 0$}\rangle,\\
U_1&:=\langle U_\alpha\mid\alpha\in\Phi_i\mbox{ for all $\geq 1$}\rangle,\\
L&:=\langle T,U_\alpha\mid\alpha\in\Phi_0\rangle.
\end{align*}
Then $P$ is a parabolic subgroup of $G$ with unipotent radical $U_1$ and
Levi decomposition $P=U_1\rtimes L$. The Lie algebra of $P$ is given by 
$\fp:=\text{Lie}(P)=\fg(\geq 0)\subseteq \fg$. More generally, for any 
integer $i\geq 1$, we set 
\[ U_i:=\langle U_\alpha\mid  \alpha\in\Phi_j \mbox{ for all $j\geq i$} 
\rangle \subseteq G.\]
Thus, we obtain a chain of subgroups $P\supseteq U_1\supseteq U_2
\supseteq U_3\supseteq \ldots$; using Chevalley's commutator relations, 
one immediately sees that each $U_i$ is a normal subgroup of $P$
and that $U_i/U_{i+1}$ is abelian.
\end{abs}

\begin{abs} \label{abs15} Let us fix a weighted Dynkin diagram $d\in \Delta$ 
as above. For any integer $i\geq 1$, we have a corresponding 
subgroup $U_i\subseteq G$ and a corresponding subspace $\fg(i)\subseteq \fg$.
Following Kawanaka \cite[(3.1.1)]{kaw1}, we define a map  
\[f\colon U_1\rightarrow \fg(1)\oplus \fg(2)\]
as follows. Let $u\in U_1$. As in \ref{abs12}, we have a unique 
expression 
\[u =\prod_{\alpha\in\Phi_i \text{ for $i\geq 1$}} x_\alpha(u_\alpha) 
\qquad (u_\alpha\in k)\]
where the product is taken in the given order $\preceq$ on $\Phi^+$.
Then we set 
\[f(u)=f\Bigl(\prod_{\alpha\in\Phi_i \text{ for $i\geq 1$}} x_\alpha
(u_\alpha)\Bigr):=\sum_{\alpha\in\Phi_1\cup\Phi_2} u_\alpha e_\alpha.\]
\end{abs}

\begin{lem}[Cf.\ Kawanaka \protect{\cite[\S 3.1]{kaw1}}] \label{abs16}
Let $u,v\in U_1$. Then the following hold.
\begin{itemize}
\item[(a)] If $u\in U_2$ or $v\in U_2$, then $f(uv)=f(u)+f(v)$.
\item[(b)] $f(uvu^{-1}v^{-1})\equiv [f(u),f(v)] \bmod \fg(\geq 3)$.
\item[(c)] $f(F(u))=F(f(u))$.
\end{itemize}
\end{lem}

\begin{proof} This is a rather straightforward application of Chevalley's 
commutator relations. For (b), we use the fact that $[e_\alpha,e_\beta]=
C_{\alpha,\beta,1,1} e_{\alpha+\beta}$ if $\alpha+\beta\in\Phi$; see 
\ref{abs12}. For (c), we use the formulae in \ref{abs13}. We omit
further details.
\end{proof}


\begin{abs} \label{abs17} The general idea for defining GGGRs corresponding 
to a fixed $d\in \Delta$ is as follows. (In the following discussion we
avoid any reference to the characteristic of~$k$.) First of all, we assume 
that $d$ is invariant under the permutation $\tau\colon \Phi\rightarrow\Phi$
induced by $F$. Consequently, all the subgroups $P$, $U_i$ ($i\geq 1$) of $G$
are $F$-stable and all the subspaces $\fg(i)$ ($i\geq 0$) are $F$-stable.
Let us fix a non-trivial character $\psi\colon \F_q^+\rightarrow \C^\times$. 

Let us also consider a linear map $\lambda\colon \fg(2)\rightarrow k$ 
defined over $\F_q$, that is, we have $\lambda(F(y))=\lambda(y)^q$ for all 
$y\in \fg(2)$. Then Lemma~\ref{abs16}(a) shows that $U_2\rightarrow
k^+$, $u\mapsto \lambda(f(u))$, is a group homomorphism and so, by 
Lemma~\ref{abs16}(c), we also obtain a group homomorphism 
\[ \chi_\lambda \colon U_2^F \rightarrow \C^\times, \qquad u\mapsto
\psi\bigl(\lambda(f(u))\bigr).\]
We shall require that $\lambda$ is in ``sufficiently general position'' 
(where this term will have to be further specified; see 
Definition~\ref{def24} below). Let us assume that this is the case. 
If $\fg(1)=\{0\}$, then the GGGR corresponding to~$d, \lambda$ will simply 
be given by the induced representation  
\[\Gamma_{d,\lambda}:=\Ind_{U_2^F}^{G^F}\bigl(\chi_\lambda\bigr).\]
Now consider the case where $\fg(1)\neq \{0\}$. Since $[\fg(1),\fg(1)] 
\subseteq \fg(2)$, we obtain a well-defined alternating bilinear form 
\[ \sigma_\lambda \colon \fg(1)\times \fg(1)\rightarrow k,\qquad (y,z)\mapsto
\lambda\bigl([y,z]\bigr).\]
Assume also that the radical of this bilinear form is zero. Then we choose 
an $F$-stable Lagrangian subspace in $\fg(1)$ and pull back this subspace 
to an $F$-stable subgroup $U_{1.5} \subseteq U_1$ via the map~$f$.
Using Lemma~\ref{abs16}(b) we see that $\ker(\chi_\lambda)$ is normal 
in $U_{1.5}^F$ and $U_{1.5}^F/\ker(\chi_\lambda)$ is an abelian $p$-group. 
(See also the proof of \cite[Lemma~3.1.9]{kaw1}.) So we can extend 
$\chi_\lambda$ to a character $\tilde{\chi}_\lambda \colon U_{1.5}^F 
\rightarrow \C^\times$. In this case, the GGGR corresponding 
to~$d,\lambda$ will be given by the induced representation 
\[\Gamma_{d,\lambda}:=\Ind_{U_{1.5}^F}^{G^F}\bigl(\tilde{\chi}_\lambda\bigr).\]
Note that $[U_1^F:U_{1.5}^F]=[U_{1.5}^F:U_2^F]$. Furthermore, it turns out 
that
\[ \Ind_{U_2^F}^{G^F} \bigl(\chi_\lambda \bigr)=[U_1^F:U_{1.5}^F]\cdot 
 \Gamma_{d,\lambda},\]
which shows that the definition of $\Gamma_{d,\lambda}$ does not depend on 
the choice of the Lagrangian subspace or the extension 
$\tilde{\chi}_\lambda$ of~$\chi_\lambda$ (cf.\ \cite[1.3.6]{kaw0}, 
\cite[3.1.12]{kaw1}).
\end{abs}

In Kawanaka's set-up \cite[\S 1.2]{kaw0}, \cite[\S 3.1]{kaw1}, the above
assumption on the radical of $\sigma_\lambda$ is always satisfied. (See
also Remark~\ref{rem24a} below.) Our plan for the definition of GGGRs in 
bad characteristic is to follow the above general procedure, but we have 
to find out in which situations this still makes sense at all. The 
following two examples show that there is a serious issue concerning 
the radical of $\sigma_\lambda$ when $\fg(1)\neq \{0\}$.
 
\begin{exmp} \label{bsp4} Let $G=\mbox{Sp}_4(k)$. 
We have $\Phi=\{\pm \alpha,\pm \beta,\pm (\alpha+\beta),\pm (2\alpha+
\beta)\}$ where $\Pi=\{\alpha,\beta\}$; here, $\alpha$ is a short simple
root and $\beta$ is a long simple root. By \cite[p.~394]{C2}, there are 
$4$~weighted Dynkin diagrams $d\in \Delta$, where:
\[ (d(\alpha),d(\beta))\in\{(0,0),(1,0),(0,2),(2,2)\}.\]
Let $d_0\in\Delta$ be such that $d_0(\alpha)=1$ and $d_0(\beta)=0$. Then
$\cb_{d_0}=2$ and 
\[ \fg(1)=\langle e_\alpha,e_{\alpha+\beta}\rangle_k, \qquad\qquad
\fg(2)=\langle e_{2\alpha+\beta}\rangle_k.\]
We have $[e_\alpha,e_{\alpha+\beta}]=\pm 2e_{2\alpha+\beta}$. Let
$\lambda \colon \fg(2)\rightarrow k$ be a linear map. If $p\neq 2$, then 
the radical of the alternating form $\sigma_\lambda$ is zero whenever 
$\lambda(e_{2\alpha+ \beta})\neq 0$. Now assume that $p=2$. Then 
$[e_\alpha,e_{\alpha+\beta}]=0$ and so $\sigma_\lambda$ is identically zero
for any~$\lambda$. The commutator relations in \ref{abs12} also show that 
the subgroup $U_1=\langle U_\alpha, U_{\alpha+\beta}, U_{2\alpha+\beta}
\rangle$ associated with~$d_0$ is abelian. In this case, it is not clear to 
us at all how one should proceed in order to define a GGGR associated 
with~$d_0$.
\end{exmp}

\begin{abs} \label{antid} To simplify the notation for matrices,
we define $\mbox{antidiag}(x_1,\ldots,x_n)$ to be the $n\times n$-matrix
with entry $x_i$ at position $(i,n+1-i)$ for $1\leq i \leq n$, and
entry $0$ otherwise. Thus, for example,
\[ \mbox{antidiag}(x_1,x_2,x_3)=\left(\begin{array}{ccc} 0 & 0 & x_1\\
0 & x_2 & 0 \\ x_3 & 0 & 0 \end{array}\right).\]
\end{abs}

\begin{exmp} \label{bg2} Let $G=G_2(k)$. We have 
\[\Phi=\{\pm \alpha,\pm \beta,\pm (\alpha+\beta),\pm (2\alpha+ \beta),
\pm (3\alpha+\beta),\pm (3\alpha+2\beta)\}\]
where $\Pi=\{\alpha,\beta\}$; here, $\alpha$ is a short simple
root and $\beta$ is a long simple root. By \cite[p.~401]{C2}, there are 
$5$~weighted Dynkin diagrams $d\in \Delta$, where:
\[ (d(\alpha),d(\beta))\in\{(0,0),(0,1),(1,0),(0,2),(2,2)\}.\]

(a) Let $d\in\Delta$ be such that $d(\alpha)=1$ and $d(\beta)=0$. Then
$\cb_d=2$ and 
\[ \fg(1)=\langle e_\alpha,e_{\alpha+\beta}\rangle_k, \qquad\qquad
\fg(2)=\langle e_{2\alpha+\beta}\rangle_k.\]
We have $[e_\alpha,e_{\alpha+\beta}]=\pm 2e_{2\alpha+\beta}$. Let $\lambda
\colon \fg(2)\rightarrow k$ be a linear map. If $p=2$, then $\sigma_\lambda$ 
is identically zero for any~$\lambda$. If $p\neq 2$, then the radical of 
$\sigma_\lambda$ is zero whenever $\lambda(e_{2\alpha+\beta})\neq 0$.

(b) Let $d\in\Delta$ be such that $d(\alpha)=0$ and $d(\beta)=1$. Then
$\cb_d=3$ and 
\[ \fg(1)=\langle e_\beta,e_{\alpha+\beta},e_{2\alpha+\beta},
e_{3\alpha+\beta}\rangle_k, \qquad \qquad \fg(2)=\langle 
e_{3\alpha+2\beta}\rangle_k.\]
The only non-zero Lie brackets are $[e_\beta,e_{3\alpha+\beta}]=
\pm e_{3\alpha+2\beta}$, $[e_{\alpha+\beta},e_{2\alpha+\beta}]=\pm 3
e_{3\alpha+2\beta}$. Hence, if $\lambda\colon \fg(2)\rightarrow k$ is any
linear map, then the Gram matrix of $\sigma_\lambda$ with respect to
the above basis of $\fg(1)$ is given by
\begin{center}
$\pm x_1\cdot\mbox{antidiag}(1,3,-3,-1)\qquad\mbox{where}\qquad 
x_1:=\lambda(e_{3\alpha+2\beta})$. 
\end{center}
The determinant of this matrix is $\pm 9x_1^4$. Hence, if $p=3$, then 
there is no $\lambda$ such that the radical of $\sigma_\lambda$ is zero.
On the other hand, if $p\neq 3$, then the radical of $\sigma_\lambda$ 
is zero whenever $x_1=\lambda(e_{3\alpha+2\beta})\neq 0$.
\end{exmp}



\section{Nilpotent and unipotent pieces} \label{sec2}

We keep the set-up of the previous section. Given a weighted Dynkin
diagram $d\in\Delta$, our main task is to find suitable conditions
under which a linear map $\lambda\colon \fg(2)\rightarrow k$ may be 
considered to be in ``sufficiently general position'' (cf.\ \ref{abs17}).
For this purpose, we use Lusztig's framework \cite{L3a}--\cite{L3d} for
dealing with unipotent elements in $G$ and nilpotent elements in $\fg$
when~$p$ (the characteristic of~$k$) is small.

\begin{abs} \label{abs20} Let $\cU$ be the variety of unipotent elements
of $G$. In \cite[\S 1]{L3a}, Lusztig introduced a natural partition 
\[ \cU=\coprod_{d \in \Delta} H_d\]
where each $H_d$ is an irreducible locally closed subset stable under
conjugation by $G$. A general, case-free proof for the existence of this
partition was given by Clarke--Premet \cite[Theorem~1.4]{clpr}. The sets 
$\{H_d\mid d\in \Delta\}$ are called the {\it unipotent pieces} of $G$. In 
each such piece $H_d$, there is a unique unipotent class $C_d$ of $G$ such 
that $C_d$ is open dense in $H_d$. If $p$ is a good prime for $G$, then 
$H_d=C_d$. In general, $H_d$ is the union of $C_d$ and a finite number of 
unipotent classes of dimension strictly smaller than $\dim C_d$. We will 
say that the unipotent classes $\{C_d\mid d\in \Delta\}$ ``come from 
characteristic~$0$''. (Alternatively, the latter notion can be defined 
using the Springer correspondence, see \cite[1.3, 1.4]{L3a}, or the 
results of Spaltenstein \cite{spa}; see also \cite[\S 2]{gema}. All
these definitions agree as can be checked using the explicit knowledge
of the unipotent classes and the Springer correspondence in all cases.)
\end{abs}


\begin{abs} \label{abs21} We recall some further notation and some results 
from \cite[\S 2]{L3d}. There is a coadjoint action of $G$ on the dual vector
space $\fg^*$ which we denote by $g.\xi$ for $g\in G$ and $\xi\in \fg^*$; 
thus, $(g.\xi)(y)=\xi(\mbox{Ad}(g^{-1})(y))$ for all $y\in \fg$. We denote
by $G_\xi$ the stabilizer of $\xi\in\fg^*$ under this action. As in 
\cite{L3d}, an element $\xi\in \fg^*$ is called nilpotent if there 
exists some $g\in G$ such that the Lie algebra of the Borel subgroup 
$B\subseteq G$ is contained in $\mbox{Ann}(g.\xi)$. Let 
\[\cN_{\fg^*}:=\{\xi\in\fg^*\mid \mbox{$\xi$ is nilpotent}\}.\]
For any $Y\subseteq \fg$, we denote $\mbox{Ann}(Y):=\{\xi\in
\fg^*\mid \xi(y)=0 \mbox{ for all $y\in Y$}\}$. Let us fix a weighted 
Dynkin diagram $d\in \Delta$. As in \ref{abs14a}, we have a corresponding
grading $\fg=\bigoplus_{i\in \Z} \fg(i)$. In order to indicate the 
dependence on $d$, we shall now write $\fg_d(i)=\fg(i)$ for all $i\in\Z$;
similarly, we write $P_d=P$ for the corresponding parabolic subgroup of~$G$.
Now, we also have a grading $\fg^*=\bigoplus_{j\in\Z}\fg_d(j)^*$ where we set 
\[\fg_d(j)^*:=\mbox{Ann}\Bigl(\bigoplus_{i\in \Z:\;i\neq -j} \fg_d(i)\Bigr)
\qquad \mbox{for any $j\in\Z$}.\] 
We note that the subspace $\fg_d(\geq j)^*:=\bigoplus_{j'\in\Z:\;j'\geq j} 
\fg_d(j')^*$ is stable under the coadjoint action of $P_d$. Let 
\[ \fg_d(2)^{*!}:=\{\xi\in\fg_d(2)^*\mid G_\xi\subseteq P_d\}\]
and $\sigma_d^*:=\fg_d(2)^{*!}+\fg_d(\geq 3)^*\subseteq \fg_d(\geq 2)^*$. 
Then $\sigma_d^*$ is stable under the coadjoint action of $P_d$ on 
$\fg_d(\geq 2)^*$. Finally, let $\hat{\sigma}_d^* \subseteq \fg^*$ be 
the union of the orbits of the elements in $\sigma_d*$ under the coadjoint 
action of $G$. Then $\xi\mapsto\xi$ is a map 
\[ \Psi_{\fg^*}\colon \coprod_{d\in\Delta} \hat{\sigma}_d^*\quad\rightarrow
\quad \cN_{\fg^*}.\]
By \cite[Theorem~2.2]{L3d}, the map $\Psi_{\fg^*}$ is a bijection if
the adjoint group of $G$ is a direct product of simple groups of types
$A$, $C$ and $D$. By the main result of \cite{xue1}, this also holds
if there is a direct factor of type $B$. In the remarks just following
\cite[Theorem~2.2]{L3d}, Lusztig expresses the expectation that 
$\Psi_{\fg^*}$ is a bijection without any restriction on $G$.
By Clarke--Premet \cite[Theorem~7.3]{clpr} it is known that 
$\Psi_{\fg^*}$ is always surjective.
\end{abs}


\begin{abs} \label{abs23} In order to apply the above results to the 
situation in Section~\ref{sec1}, we need a mechanism by which we can pass 
back and forth between the vector spaces $\fg_d(i)$ and $\fg_d(i)^*$. If 
there exists a $G$-equivariant vector space isomorphism $\fg
\stackrel{\sim}{\rightarrow} \fg^*$, then there is a canonical way
of doing this, as explained in \cite[2.3]{L3d}. However, such
an isomorphism will not always exist. To remedy this situation, we
follow Kawanaka \cite[\S 1.2]{kaw0}, \cite[\S 3.1]{kaw1} and fix 
an $\F_q$-opposition automorphism $\fg\rightarrow \fg$, $y\mapsto 
y^\dagger$. This is a linear isomorphism, defined over $\F_q$, such that
$\ft^\dagger=\ft$ and $e_\alpha^\dagger=\pm e_{-\alpha}$ for all 
$\alpha\in\Phi$. (See also \cite[Lemma~5.2]{tay}.)  If $\xi\in \fg^*$,
then we define $\xi^\dagger\in\fg^*$ by $\xi^\dagger(y):=\xi(y^\dagger)$
for $y\in\fg$.
\end{abs}

\begin{defn} \label{def24} Let $d\in\Delta$ be a weighted Dynkin diagram
and consider the corresponding grading $\fg=\bigoplus_{i\in \Z}\fg_d(i)$.
Let $\lambda\colon \fg_d(2)\rightarrow k$ be a linear map. We regard
$\lambda$ as an element of $\fg^*$ by setting $\lambda$ equal to zero
on $\fg_d(i)$ for all $i\neq 2$. We say that $\lambda$ is in 
``sufficiently general position'' if the following conditions hold.
\begin{itemize}
\item[(K1)] We require that $\lambda^\dagger \in\fg_d(2)^{*!}$, that is,
$G_{\lambda^\dagger}\subseteq P_d$; see \ref{abs21}, \ref{abs23}.
\item[(K2)] If $\fg_d(1)\neq \{0\}$, then we also require that the radical of 
the corresponding alternating form $\sigma_\lambda\colon \fg_d(1)\times
\fg_d(1)\rightarrow k$ in \ref{abs17} is zero.
\end{itemize}
Note that (K1), (K2) only refer to the algebraic group $G$, but not to the
Frobenius map~$F$. If (K1), (K2) hold and if $\lambda$ is defined over 
$\F_q$, then we can follow the procedure in \ref{abs17} and define the 
corresponding GGGR $\Gamma_{d,\lambda}$ of the finite group~$G^F$.
\end{defn}

\begin{rem} \label{rem24a} Kawanaka's work \cite{kaw0}, \cite{kaw1}
fits into this setting as follows. Assume that $p$ is a good prime for $G$ 
and that there exists a non-degenerate, symmetric and $G$-invariant bilinear 
form $\kappa\colon \fg\times \fg\rightarrow k$. We also need to make a 
certain technical assumption on the isogeny type of the derived subgroup 
of~$G$. (For details see \cite[3.22]{tay}, \cite{prem}.) Let $d\in \Delta$. 
Then there is a dense open orbit under the adjoint action of $P_d$ on 
$\fg_d(2)$. Let $e$ be an element of this orbit and define a linear map 
$\lambda_e\colon \fg_d(2) \rightarrow k$ as follows.
\[ \lambda_e(y):=\kappa(e^\dagger,y) \qquad \mbox{for $y\in \fg_d(2)$}.\]
Then (K1), (K2) are satisfied for $\lambda=\lambda_e$; see 
\cite[\S 1.2]{kaw0}, \cite[\S 3.1]{kaw1}. Furthermore, if $d$ is 
invariant under the permutation of $\Phi$ induced by~$F$, then $e$
can be chosen such that (K1), (K2) hold and $\lambda_e$ is defined
over~$\F_q$. For example, all this holds for $G=\SL_n(k)$ with no 
restriction on~$p$; see \cite[1.2]{kaw0}.
\end{rem}

\begin{rem} \label{rem24b} As already mentioned above, the map 
$\Psi_{\fg^*}$ in \ref{abs21} is always surjective. More precisely, given 
$d\in\Delta$, the subset $\fg_d(2)^{*!}\subseteq \fg_d(2)^*$ is non-empty. 
By \cite[Remark~1 (p.~665)]{clpr}, this subset actually contains a dense 
open subset of $\fg_d(2)^*$ (denoted by $X^{\triangle}(\fg^*)$ in 
\cite[7.1]{clpr}) and so $\fg_d(2)^{*!}$ itself is a dense subset of 
$\fg_d(2)^*$. Thus, there always exists a dense set of linear maps 
$\lambda\colon \fg_d(2)\rightarrow k$ such that condition (K1) in 
Definition~\ref{def24} is satisfied. 
\end{rem}

As illustrated by the examples at the end of Section~\ref{sec1}, the 
condition (K2) requires more attention. 

\begin{rem} \label{rem24c} Let $d\in \Delta$ and assume that $\fg_d(1)\neq
\{0\}$. Let $\Phi_1=\{\beta_1,\ldots, \beta_n\}$ and $\Phi_2=\{\gamma_1,
\ldots,\gamma_m\}$. Given a linear map $\lambda \colon \fg_d(2)\rightarrow 
k$, we denote by $\cG_\lambda \in M_n(k)$ the Gram matrix of 
$\sigma_\lambda$ with respect to the basis $\{e_{\beta_1},\ldots,
e_{\beta_n}\}$ of $\fg_d(1)$. The entries of $\cG_\lambda$ are given as
follows. We set $x_l:=\lambda(e_{\gamma_l})$ for $1\leq l \leq m$. For 
$1\leq i,j\leq n$, we define an element $\nu_{ij} \in k$ as follows. 
If $\beta_i+\beta_j\not\in\Phi$, then $\nu_{ij}:=0$. Otherwise, there is a 
unique $l(i,j)\in\{1,\ldots,m\}$ and some $\nu_{ij} \in k$ such that 
$[e_{\beta_i},e_{\beta_j}]=\nu_{ij}e_{\gamma_{l(i,j)}}$. Then we have
\[ (\cG_\lambda)_{ij}=\sigma_\lambda(e_{\beta_i},e_{\beta_j})=
\left\{\begin{array}{cl} x_{l(i,j)} \nu_{ij} & \qquad \mbox{if $\beta_i+
\beta_j\in \Phi$},\\ 0&\qquad \mbox{otherwise}.  \end{array}\right.\]
In order to work this out explicity, we may assume without loss of generality
that $G$ is semisimple (since $U_\alpha\subseteq \Gder$ for all $\alpha\in
\Phi$; see \ref{abs11}). But then, by \cite[Remark (p.~64)]{St}, the 
structure constants of the Lie algebra $\fg$ are obtained from those of a 
Chevalley basis of $\fg_0$ by reduction modulo~$p$. Thus, we can explicitly 
determine the elements $\nu_{ij}$, via a computation inside~$\fg_0$. Using 
one of the two canonical Chevalley bases in \cite[\S 5]{myg} (the two bases
only differ by a global sign), one can even avoid the issue of choosing 
certain signs. Hence, there is a purely combinatorial algorithm for 
computing $\cG_\lambda$, and this can be easily implemented in {\sf GAP} 
\cite{gap4}. In particular, we have: 
\begin{itemize}
\item[($*$)] Up to a global sign, the Gram matrix $\cG_\lambda$ only 
depends on the root system $\Phi$ and the values 
$\lambda(e_\alpha)$ ($\alpha\in\Phi_2$).
\end{itemize}
The radical of $\sigma_\lambda$ is zero if and only if
$\det(\cG_\lambda)\neq 0$. Now we notice that this determinant is given by 
evaluating a certain $m$-variable polynomial at $x_1,\ldots,x_m$. In 
particular, we see that condition (K2) is an ``open'' condition: either
there is no $\lambda$ at all for which (K2) holds, or (K2) holds for a 
non-empty open set of linear maps $\lambda\colon \fg_d(2)\rightarrow k$. 
Combining this with the discussion concerning (K1) in Remark~\ref{rem24b}, 
we immediately obtain the following conclusion.
\end{rem}

\begin{cor} \label{cor24} Let $d\in \Delta$ and assume that $\fg_d(1) \neq 
\{0\}$. Then either there is a non-empty open set of linear maps $\lambda 
\colon \fg_d(2) \rightarrow k$ in ``sufficiently general position'', or 
there is no such linear map at all. 
\end{cor}

With these preparations, we now obtain our first example where bad primes
exist but (K2) holds without any restriction on the field~$k$.

\begin{exmp} \label{bd4} Let $G$ be of type $D_4$. Let $\Pi=\{\alpha_1,
\alpha_2,\alpha_3,\alpha_4\}$ where $\alpha_1,\alpha_2,\alpha_4$ are all 
connected to $\alpha_3$. By \cite[p.~396--397]{C2}, there are $12$ weighted
Dynkin diagrams $d\in \Delta$. There are two of them with $\fg_d(1)\neq\{0\}$.

(a) Let $d(\alpha_1)=d(\alpha_2)=d(\alpha_4)=0$ and $d(\alpha_3)=1$. 
We have $\cb_d=7$ and
\begin{align*}
\fg_d(1)&=\langle e_{\alpha_3}, e_{\alpha_1+\alpha_3},
e_{\alpha_2+\alpha_3}, e_{\alpha_3+\alpha_4},
e_{\alpha_1+\alpha_2+\alpha_3},\\&\qquad\quad
e_{\alpha_1+\alpha_3+\alpha_4},
e_{\alpha_2+\alpha_3+\alpha_4},
e_{\alpha_1+\alpha_2+\alpha_3+\alpha_4}\rangle_k,\\
\fg_d(2)&=\langle e_{\alpha_1+\alpha_2+2\alpha_3+\alpha_4}\rangle_k.
\end{align*}
Let $\lambda\colon \fg_d(2)\rightarrow k$ be any linear map. As explained
in Remark~\ref{rem24c}, we can work out the Gram matrix $\cG_\lambda$ 
of the alternating form $\sigma_\lambda$. It is given by
\begin{center}
$\cG_\lambda=\pm\mbox{antidiag}(-x_1,x_1,x_1,x_1,-x_1,-x_1,-x_1,x_1)$.
\end{center}
where we set $x_1:=\lambda(e_{\alpha_1+\alpha_2+2\alpha_3+\alpha_4})$.
If $x_1\neq 0$, then $\det(\cG_\lambda)\neq 0$ and so the radical of 
$\sigma_\lambda$ is zero. Hence, condition (K2) is satisfied for such 
choices of~$\lambda$, and this works for any field~$k$.


(b) Let $d(\alpha_1)=d(\alpha_2)=d(\alpha_4)=1$ and $d(\alpha_3)=0$. We have
$\cb_d=4$ and 
\begin{align*}
\fg_d(1)&= \langle e_{\alpha_1},e_{\alpha_2},e_{\alpha_4},e_{\alpha_1+
\alpha_3},e_{\alpha_2+\alpha_3},e_{\alpha_3+\alpha_4}\rangle_k,\\ 
\fg_d(2)&=\langle e_{\alpha_1+\alpha_2+\alpha_3},e_{\alpha_1+\alpha_3+
\alpha_4},e_{\alpha_2+\alpha_3+\alpha_4}\rangle_k.
\end{align*}
Again, let $\lambda\colon \fg_d(2)\rightarrow k$ be any linear map. 
As above, we now obtain 
\begin{center} {
$\cG_\lambda=\pm\left(\begin{array}{r@{\hspace{5pt}}r@{\hspace{5pt}}
r@{\hspace{5pt}}r@{\hspace{5pt}}r@{\hspace{5pt}}r}
   0& 0& 0& 0& x_1& x_2 \\  0& 0& 0& x_1& 0& x_3 \\  0& 0& 0& x_2& x_3& 0 \\ 
   0& -x_1& -x_2& 0& 0& 0 \\  -x_1& 0& -x_3& 0& 0& 0 \\ 
   -x_2& -x_3& 0& 0& 0& 0 \end{array}\right).$}
\end{center}
where $x_1:=\lambda(e_{\alpha_1+\alpha_2+\alpha_3})$,
$x_2:=\lambda(e_{\alpha_1+\alpha_3+\alpha_4})$,
$x_3:=\lambda(e_{\alpha_2+\alpha_3+\alpha_4})$.
We compute $\det(\cG_\lambda)=4x_1^2x_2^2x_3^2$. Hence, if $p=2$, then the 
radical of $\sigma_\lambda$ will never be zero and so condition (K2) 
will never be satisfied. On the other hand, if $p\neq 2$, then the 
radical of $\sigma_\lambda$ will be zero whenever $x_1x_2x_3\neq 0$.
\end{exmp}

In the above examples, it was easy to compute the determinant of the 
Gram matrix $\cG_\lambda$. However, we will encounter examples below 
where the computation of $\det(\cG_\lambda)$ becomes a very serious issue.

\section{Unipotent support} \label{sec3}

We are now looking for a unifying principle behind the various examples
that we have seen so far. In Conjecture~\ref{main} below, we propose such
a unification. First we need some preparations.

\begin{abs} \label{abs31} Let $\Irr(G^F)$ be the set of complex 
irreducible representations of $G^F$ (up to isomorphism). 
Then there is a canonical map 
\[ \Irr(G^F)\rightarrow \{\text{$F$-stable unipotent classes of $G$}\},
\qquad \rho\mapsto C_\rho,\]
defined in terms of the notion of ``unipotent support''. To explain this,
we need to introduce some notation. Let $C$ be an $F$-stable unipotent 
conjugacy class of $G$. Then $C^F$ is a union of conjugacy classes of $G^F$. 
Let $u_1, \ldots,u_r\in C^F$ be representatives of the classes of $G^F$ 
contained in $C^F$. For $1\leq i \leq r$ we set $A(u_i):=C_G(u_i)/C_G^\circ
(u_i)$. Since $F(u_i)=u_i$, the Frobenius map $F$ induces an automorphism of 
$A(u_i)$ which we denote by the same symbol. Let $A(u_i)^F$ be the group of 
fixed points under $F$. Then we set
\[ \AV(\rho,C):=\sum_{1 \leq i \leq r} |A(u_i):A(u_i)^F|\mbox{trace}\bigl(
\rho(u_i)\bigr)\]
for any $\rho\in\Irr(G^F)$. Note that this does not depend on the choice 
of the representatives $u_i$. Now the desired map is obtained as follows. 
Let $\rho\in\Irr(G^F)$ and set $a_\rho:=\max\{ \dim C\mid \AV(\rho, C)
\neq 0\}$ (where the maximum is taken over all $F$-stable unipotent classes
$C$ of~$G$). By the main results of \cite{gema}, \cite{L2}, there is a 
unique $C$ such that $\dim C=a_\rho$ and $\AV(\rho,C)\neq 0$. This $C$ 
will be denoted by $C_\rho$ and called the {\it unipotent support} 
of~$\rho$. 

By \cite[Remark~3.9]{gema}, it is known that $C_\rho$ comes 
from characteristic~$0$ (see \ref{abs20}) and, hence, equals $C_{d_\rho}$ 
for a well-defined weighted Dynkin diagram $d_\rho\in\Delta$. We set
\[ \Delta_{k,F}^{\!\bullet}:=\{d_\rho\in \Delta\mid \rho\in \Irr(G^F)\}.\]
Thus, $\Delta_{k,F}^{\!\bullet}$ consists precisely of those weighted Dynkin 
diagrams for which the corresponding unipotent class of $G$ occurs as the 
unipotent support of some irreducible representation of $G^F$. In order to 
obtain a subset of $\Delta$ which only depends on $G$ and not on the choice 
of the particular Frobenius map $F$, we set 
\[ \Delta_k^{\!\bullet}:=\bigcup_{n\geq 1} \Delta_{k,F^n}^{\!\bullet}.\]
(Note that, if $F_1\colon  G\rightarrow G$ is another Frobenius map, then
there always exist integers $n,n_1\geq 1$ such that $F_1^{n_1}=F^n$.)
Thus, $\cC^\bullet=\{C_d \mid d\in \Delta_k^{\!\bullet}\}$ is precisely
the set of unipotent classes mentioned at the end of Section~\ref{sec0}.
\end{abs}

\begin{table}[htbp] \caption{The sets $\Delta_k^{\!\bullet}\setminus
\Delta_{\operatorname{spec}}$ for $G$ of exceptional type}
\label{tabb}
\begin{center} {\small
$\begin{array}{ccc} \hline G_2 & \cb_d & \text{condition}  \\ \hline
A_1  &3 & p\neq 3\\ \tilde{A}_1  &2 & p\neq 2\\\hline\\ 
\hline F_4 & \cb_d & \text{condition}  \\ \hline
A_1  &16 & p\neq 2\\ A_2{+}\tilde{A}_1  &7 & p\neq 2\\ B_2  &6 
& p\neq 2\\ \tilde{A}_2{+}A_1  & 6 & p\neq 3\\ C_3(a_1)  & 5 & 
p\neq 2 \\\hline\\ \hline E_6  & \cb_d & \text{condition}  \\ \hline
3A_1 & 16 & p\neq 2\\2A_2{+}A_1  &  9 & p\neq 3\\A_3{+}A_1  & 8 
& p\neq 2\\ A_5 & 4 & p\neq 2\\\hline\\
\hline E_7 & \cb_d & \text{condition}  \\ \hline
(3A_1)' & 31 & p\neq 2\\ 4A_1 & 28 & p\neq 2\\ 2A_2{+}A_1 & 18 & p\neq 3\\
(A_3{+}A_1)' & 17 & p\neq 2\\ A_3{+}2A_1 & 16 & p\neq 2\\ D_4{+}A_1 
&12& p\neq 2\\ A_5' &9 & p\neq 2\\ A_5{+}A_1 &9 & p\neq 3\\ D_6(a_2) &8 & 
p\neq 2\\ D_6 & 4& p\neq 2\\\hline
\end{array}\qquad\quad \begin{array}{cccc} 
\hline E_8 & \cb_d & \text{condition}  \\ \hline
3A_1 & 64 & p\neq 2\\
4A_1 & 56 & p\neq 2\\
A_2{+}3A_1 & 43 & p\neq 2\\
2A_2{+}A_1 & 39 & p\neq 3\\
A_3{+}A_1 & 38 & p\neq 2\\
2A_2{+}2A_1 & 36 & p\neq 3\\
A_3{+}2A_1 & 34 & p\neq 2\\
A_3{+}A_2{+}A_1 & 29 & p\neq 2\\
D_4{+}A_1 & 28 & p\neq 2\\
2A_3 & 26 & p\neq 2\\
A_5 & 22 & p\neq 2\\
A_4{+}A_3 & 20 & p\neq 5\\
A_5{+}A_1 & 19 & p\neq 2,3\\
D_5(a_1){+}A_2 & 19 & p\neq 2\\
D_6(a_2) & 18 & p\neq 2\\
E_6(a_3){+}A_1 & 18 & p\neq 3\\
E_7(a_5) & 17 & p\neq 2\\
D_5{+}A_1 & 16 & p\neq 2\\
D_6 & 12 & p\neq 2\\
A_7 & 11 & p\neq 2\\
E_6{+}A_1 & 9 & p\neq 3\\
E_7(a_2) & 8 & p\neq 2\\
D_7 & 7 & p\neq 2\\
E_7 & 4 & p\neq 2\\ \hline \multicolumn{3}{l}{\text{(Notation from 
\cite[\S 13.1]{C2})}}\\\\\\
\end{array}$}
\end{center}
\end{table}

\begin{rem} \label{uni1} Recall from Lusztig \cite{Ls1}, \cite{Ls2} the 
notion of {\it special unipotent classes}. The precise definition of these 
classes is not elementary; it involves the Springer correspondence and the 
notion of {\it special representations} of the Weyl group of~$G$. By 
\cite[Prop.~4.2]{gema}, an $F$-stable unipotent class is special if and 
only if it is the unipotent support of some unipotent representation 
of $G^F$. Thus, we conclude that special unipotent classes come from 
characteristic~$0$ and we have:
\begin{itemize}
\item[(a)] $\Delta_{\text{spec}}\subseteq \Delta_k^{\!\bullet}$, where 
$\Delta_{\text{spec}}$ denotes the set of all $d\in \Delta$ such that the 
corresponding unipotent class $C_d$ of $G$ is special. 
\end{itemize}
Explicit descriptions of the sets $\Delta_{\text{spec}}$ are contained
in the tables in \cite[\S 13.4]{C2}. Using the above concepts, one can 
give an alternative description of the map $\rho\mapsto C_\rho$; see 
\cite[13.4]{L1}, \cite[10.9]{L2}, \cite[\S 1]{L4}, \cite[\S 3.C]{gema}. 
This alternative description allows one to compute $\Delta_k^{\!\bullet}$ 
explicitly, without knowing any character values of $G^F$. In particular, 
this yields the following statement:
\begin{itemize}
\item[(b)] If $p$ is a good prime for $G$, then $\Delta_k^{\!\bullet}=
\Delta$. 
\end{itemize}
This was first stated (in terms of the alternative description of $\rho
\mapsto C_\rho$ and for $p$ large) as a conjecture in \cite[\S 9]{Ls1}; see 
also \cite[13.4]{L1}. A full proof eventually appeared in 
\cite[Theorem~1.5]{L4}.
\end{rem}

\begin{prop} \label{uni2} Assume that $G$ is simple and $p$ is a bad prime
for $G$. If $G$ is of classical type $B_n,C_n$ or $D_n$, then 
$\Delta_k^{\!\bullet}= \Delta_{\operatorname{spec}}$. If $G$ is of 
exceptional type $G_2$, $F_4$, $E_6$, $E_7$ or $E_8$, then the sets 
$\Delta_k^{\!\bullet}\setminus \Delta_{\operatorname{spec}}$ are
specified in Table~\ref{tabb}.
\end{prop}

(In Table~\ref{tabb}, we list all $d\in \Delta\setminus 
\Delta_{\text{spec}}$; for each such $d$, the last column gives the
condition on $p$ such that $d\in \Delta_k^{\!\bullet}$.)

\begin{proof} This follows by analogous methods as in \cite{L4}. The 
special feature of the case where $G$ is of classical type and $p=2$ is
the fact that then the centraliser of a semisimple element is a Levi 
subgroup of some parabolic subgroup (see, e.g., \cite[\S 4]{C1}). If $G$ 
is of exeptional type, then one uses explicit computations completely 
analogous to those in \cite[\S 7]{L4}; here, one also needs L\"ubeck's 
tables \cite{Lue07} concerning possible centralisers of semisimple elements 
in these cases. We omit further details. 
\end{proof}

\begin{conj} \label{main} Let $d\in\Delta_k^{\!\bullet}$ be invariant
under the permutation $\tau\colon \Phi\rightarrow \Phi$ induced by~$F$. 
Then there exist linear maps $\lambda\colon\fg_d(2)\rightarrow k$ which 
are defined over $\F_q$ and are in ``sufficiently general position'' 
(see Definition~\ref{def24}). Hence, following the general procedure 
described in \ref{abs17}, we can define the corresponding 
$\operatorname{GGGR}$s $\Gamma_{d,\lambda}$ of~$G^F$. 
\end{conj}

By Remarks~\ref{rem24a}, \ref{rem24c}($*$) and \ref{uni1}(b), the conjecture 
holds for all $d\in \Delta_k^{\!\bullet}=\Delta$ if $p$ is a good prime 
for~$G$. In particular, it holds when $G$ is of type $A_n$. We will now 
discuss a number of examples supporting the conjecture in cases where~$p$ is
a bad prime. As already explained in the previous section, the main issue is 
the validity of condition (K2) in Definition~\ref{def24}.

\begin{exmp} \label{bc2bis} (a) Let $G=\mbox{Sp}_4(k)$ and assume that 
$k$ has characteristic~$2$. By Proposition~\ref{uni2}, or by inspection 
of the known character table of $G^F$ (see \cite{Eno1}), we note that the 
unipotent class corresponding to the weighted Dynkin diagram $d_0$ in 
Example~\ref{bsp4} is not the unipotent support of any irreducible 
character of~$G^F$. Thus, $d_0 \not\in \Delta_k^{\!\bullet}$ and so this 
critical case does not enter in the range of validity of 
Conjecture~\ref{main}. In fact, $\Delta_k^{\!\bullet}=\Delta\setminus
\{d_0\}$ if $p=2$.

(b) The situation is similar for $G$ of type $G_2$. Consider the two 
weighted Dynkin diagrams with $\fg_d(1)\neq \{0\}$ in Example~\ref{bg2}. 
By Table~\ref{tabb}, or by inspection of the known character table of $G^F$ 
(see \cite{EnYa} for $p=2$ and \cite{Eno2} for $p=3$), we see that 
$\Delta_k^{\!\bullet}=\Delta \setminus\{d\}$ where $d$ is as in 
Example~\ref{bg2}(a) if $p=2$, and $d$ is as in Example~\ref{bg2}(b) 
if $p=3$. 
\end{exmp}



\begin{exmp} \label{bd4bis} Let again $G$ be of type $D_4$ and return
to the discussion in Example~\ref{bd4}. The special unipotent classes
are explicitly described in \cite[p.~439]{C2}; there is only one class 
which is not special (it corresponds to elements with Jordan blocks of 
sizes $3,2,2,1$), and this is precisely the one considered in
Example~\ref{bd4}(b). But, by Proposition~\ref{uni2}, the corresponding 
weighted Dynkin diagram does not belong to $\Delta_k^{\!\bullet}$ if 
$p=2$. 
\end{exmp}

\begin{exmp} \label{be8} Let $G$ be of type $E_8$ with diagram 
\begin{center}
\begin{picture}(150,40)
\put( 1,32){$\alpha_1$}
\put( 5,25){\circle*{5}}
\put( 7,25){\line(1,0){20}}
\put( 25,32){$\alpha_3$}
\put( 29,25){\circle*{5}}
\put( 31,25){\line(1,0){20}}
\put( 49,32){$\alpha_4$}
\put( 53,25){\circle*{5}}
\put( 55,25){\line(1,0){20}}
\put( 73,32){$\alpha_5$} 
\put( 77,25){\circle*{5}}
\put( 79,25){\line(1,0){20}}
\put( 97,32){$\alpha_6$} 
\put(101,25){\circle*{5}}
\put(103,25){\line(1,0){20}}
\put(121,32){$\alpha_7$} 
\put(124,25){\circle*{5}}
\put(127,25){\line(1,0){20}}
\put(145,32){$\alpha_8$} 
\put(149,25){\circle*{5}}
\put(53,23){\line(0,-1){20}}
\put(53,3){\circle*{5}}
\put(60,2){$\alpha_2$} 
\end{picture}
\end{center}
Let $d_0\in\Delta$ correspond to the class denoted $A_4{+}A_3$ in 
Table~\ref{tabb}. We have $d_0(\alpha_4)=d_0(\alpha_7)=1$ and 
$d_0(\alpha_i)=0$ for $i\neq 4,7$; see \cite[p.~406]{C2}. 
By Table~\ref{tabb}, we have $\Delta_k^{\!\bullet}=\Delta\setminus\{d_0\}$ 
if $p=5$; furthermore, $d_0\in\Delta_k^{\!\bullet}$ if $p\neq 5$. Let 
$\lambda\colon \fg_{d_0}(2) \rightarrow k$ be a linear map and consider 
the Gram matrix $\cG_\lambda$ of the alternating form $\sigma_\lambda$. 
We claim:
\begin{itemize}
\item[(a)] If $p\neq 5$, then $\det(\cG_\lambda)\neq 0$ for some $\lambda
\colon \fg_{d_0}(2)\rightarrow k$.
\item[(b)] If $p=5$, then $\det(\cG_\lambda)=0$ for all $\lambda
\colon \fg_{d_0}(2)\rightarrow k$.
\end{itemize}
First, we find that $\dim \fg_{d_0}(1)=24$ and $\dim \fg_{d_0}(2)=
21$. As explained in Remark~\ref{rem24c}, we then explicitly work out
$\cG_\lambda$. We have 
\[\cG_\lambda=\bigl(f_{ij}(x_1,\ldots,x_{21})\bigr)_{1\leq i,j \leq 24}\]
where $f_{ij}$ are certain polynomials with integer coefficients in 
$21$ indeterminates. In order to verify (a), we argue as follows. If $p>5$, 
then (a) holds by Remark~\ref{rem24a}. If $p=2,3$, then we simply run 
through all vectors of values $(x_1,\ldots,x_{21})\in \{0,1\}^{21}$ 
(starting with the vector $1,1,\ldots,1$ and then increasing step by 
step the number of zeroes) until we find one such that $\det(\cG_\lambda)
\neq 0$. It turns out that this search is successful just after a few steps. 
The verification of (b) is much harder. Let $p=5$ and denote by $\bar{f}_{ij}$ 
the reduction of $f_{ij}$ modulo~$p$. Then we need to check that
$\det(\bar{f}_{ij})=0$. It seems to be practically impossible to compute 
such a determinant directly,  or even just the rank. (Using special values
of the $x_i$ as above, one quickly sees that the rank of $(\bar{f}_{ij})$
is at least~$22$.) Now, since $\cG_\lambda$ is anti-symmetric, we can use the 
fact that the desired determinant is given by $\mbox{Pf}(\bar{f}_{ij})^2$,
where $\mbox{Pf}(\bar{f}_{ij})$ denotes the Pfaffian of the matrix 
$(\bar{f}_{ij})$; see, for example, \cite{dress}, \cite{knu}. (I am indebted 
to Ulrich Thiel for pointing this out to me.) A simple recursive algorithm
(via row expansion, as in \cite[1.5]{dress}) is sufficient to compute 
$\mbox{Pf}(\bar{f}_{ij})=0$ in this case and yields~(b). (Over $\Z$,
the Pfaffian of $(f_{ij})$ is a non-zero polynomial which is a linear
combination of $1386$ monomials in $21$ indeterminates, where all
coefficients are divisible by~$5$.) 

The class $A_4{+}A_3$ in type $E_8$ also plays a special role in
\cite[\S 4.2]{prem2}. (I thank Alexander Premet for pointing this out
to me.)
\end{exmp} 

\begin{exmp} \label{be8b} Let again $G$ be of type $E_8$, with diagram 
as above. Let $d_1\in\Delta$ correspond to the class denoted $A_5{+}A_1$
in Table~\ref{tabb}. We have $d_1(\alpha_1)=d_1(\alpha_4)=
d_1(\alpha_8)=1$ and $d_1(\alpha_i)=0$ for $i\neq 1,4,8$; see 
\cite[p.~406]{C2}. By Table~\ref{tabb}, we have $d_1\in \Delta_k^{\!\bullet}$
if $p\neq 2,3$, and $d_1\not\in\Delta_k^{\!\bullet}$ otherwise. Let 
$\lambda\colon \fg_{d_1}(2) \rightarrow k$ be a linear map and consider 
the Gram matrix $\cG_\lambda$ of the alternating form $\sigma_\lambda$. 
Here, we find that $\dim \fg_{d_1}(1)=22$ and $\dim \fg_{d_1}(2)=
18$. Using computations as in the previous example, we obtain: 
\begin{itemize}
\item[(a)] If $p\neq 2,3$, then $\det(\cG_\lambda)\neq 0$ for some $\lambda
\colon \fg_{d_1}(2)\rightarrow k$.
\item[(b)] If $p\in\{2,3\}$, then $\det(\cG_\lambda)=0$ for all $\lambda
\colon \fg_{d_1}(2)\rightarrow k$.
\end{itemize}
(In (b), there are $\lambda$ such that $\cG_\lambda$ has rank $20$.)
\end{exmp}

The examples suggest the following characterisation of the
set $\Delta_k^{\!\bullet}$.

\begin{conj} \label{main1a} Let $d \in \Delta$. Then $d\in 
\Delta_k^{\!\bullet}$ if and only if either $\fg_d(1)=\{0\}$, or there 
exists a linear map $\lambda \colon \fg_d(2)\rightarrow k$ such that 
the radical of $\sigma_\lambda$ is zero.
\end{conj}

Finally, we state the following conjecture concerning special unipotent 
classes. As in \ref{abs14}, let $G_0$ be a connected reductive algebraic 
group over $\C$ of the same type as $G$; let $\fg_0$ be its Lie algebra. 
For $d\in \Delta$ and $i=1,2$, we set 
\[\fg_{\Z,d}(i):=\langle e_\alpha \mid d(\alpha)=i\rangle_\Z\subseteq 
\fg_0.\]
As in \ref{abs17}, given a homomorphism $\lambda \colon \fg_{\Z,d}(2)
\rightarrow \Z$, we obtain an alternating form $\sigma_\lambda \colon 
\fg_{\Z,d}(1) \times \fg_{\Z,d}(1)\rightarrow \Z$ and we may consider its 
Gram matrix with respect to the $\Z$-basis $\{e_\alpha\mid \alpha\in 
\Phi_1\}$ of $\fg_{\Z,d}(1)$. If this Gram matrix has determinant $\pm 1$, 
then we say that $\sigma_\lambda$ is non-degenerate over~$\Z$.

\begin{conj} \label{main2} With the above notation, let $d\in \Delta$.
Then we have $d\in \Delta_{\operatorname{spec}}$ if and only if either 
$\fg_{\Z,d}(1)=\{0\}$, or there exists a homomorphism $\lambda \colon
\fg_{\Z,d}(2)\rightarrow \Z$ such that $\sigma_\lambda$ is non-degenerate
over~$\Z$.
\end{conj}

Note that, if $\fg_d(1)=\{0\}$, then we certainly have $d\in 
\Delta_{\text{spec}}$. (This easily follows from \cite[Prop.~1.9(b)]{LuSp}.) 
Hence, in order to verify the above conjectures for a given example, it 
is sufficient to consider the cases where $\fg_d(1)\neq \{0\}$. This
will be further discussed in the following section.

\section{A worked example: type $F_4$} \label{sec4}

In this section, we work out in detail the example where $G$ is of type
$F_4$. We believe that the results of our computations are strong evidence 
for the truth of Conjectures~\ref{main} and~\ref{main2}; the discussion 
of the various cases will also provide a good illustration of the 
computational issues involved. Let $\Pi=\{\alpha_1,\alpha_2,\alpha_3, 
\alpha_4\}$ be a set of simple roots such that the Dynkin diagram of $G$
looks as follows:
\begin{center}
\begin{picture}(100,15)
\put( 1,10){$\alpha_1$}
\put( 5,3){\circle*{5}}
\put( 7,3){\line(1,0){20}}
\put( 29,3){\circle*{5}}
\put( 25,10){$\alpha_2$}
\put( 31,5){\line(1,0){20}}
\put( 31,2){\line(1,0){20}}
\put( 37,1){\mbox{$>$}}
\put( 49,10){$\alpha_3$}
\put( 53,3){\circle*{5}}
\put( 55,3){\line(1,0){20}}
\put( 73,10){$\alpha_4$} 
\put( 77,3){\circle*{5}}
\end{picture} 
\end{center}
By \cite[p.~401]{C2}, there are $16$~weighted Dynkin diagrams in 
$\Delta$; together with some additional information, these are printed
in Table~\ref{tab1}. (The entries in the last column are determined
by Remark~\ref{uni1}(a) and Table~\ref{tabb}.)



\begin{table}[htbp] \caption{Unipotent classes in type $F_4$}
\label{tab1}
\begin{center}
{\small $\begin{array}{ccccc} \hline \text{Name} & \text{$d\in \Delta$} 
& \cb_d & \mbox{special?} &\mbox{condition $d\in \Delta_k^{\!\bullet}$?}\\
\hline 1 & \begin{picture}(80,5)
\put( 0,0){0}
\put( 8,3){\line(1,0){10}}
\put( 20,0){0}
\put( 28,5){\line(1,0){15}}
\put( 28,2){\line(1,0){15}}
\put( 31,1){\mbox{$>$}}
\put( 45,0){0}
\put( 53,3){\line(1,0){10}}
\put( 65,0){0} \end{picture} & 24 & \text{yes} &-\\
A_1 & \begin{picture}(80,5)
\put( 0,0){1}
\put( 8,3){\line(1,0){10}}
\put( 20,0){0}
\put( 28,5){\line(1,0){15}}
\put( 28,2){\line(1,0){15}}
\put( 31,1){\mbox{$>$}}
\put( 45,0){0}
\put( 53,3){\line(1,0){10}}
\put( 65,0){0} \end{picture} & 16 & \text{no} & p\neq 2\\
\tilde{A}_1 & \begin{picture}(80,5)
\put( 0,0){0}
\put( 8,3){\line(1,0){10}}
\put( 20,0){0}
\put( 28,5){\line(1,0){15}}
\put( 28,2){\line(1,0){15}}
\put( 31,1){\mbox{$>$}}
\put( 45,0){0}
\put( 53,3){\line(1,0){10}}
\put( 65,0){1} \end{picture} & 13 & \text{yes} &-\\
A_1{+}\tilde{A}_1 & \begin{picture}(80,5)
\put( 0,0){0}
\put( 8,3){\line(1,0){10}}
\put( 20,0){1}
\put( 28,5){\line(1,0){15}}
\put( 28,2){\line(1,0){15}}
\put( 31,1){\mbox{$>$}}
\put( 45,0){0}
\put( 53,3){\line(1,0){10}}
\put( 65,0){0} \end{picture} & 10 & \text{yes} &-\\
A_2 & \begin{picture}(80,5)
\put( 0,0){2}
\put( 8,3){\line(1,0){10}}
\put( 20,0){0}
\put( 28,5){\line(1,0){15}}
\put( 28,2){\line(1,0){15}}
\put( 31,1){\mbox{$>$}}
\put( 45,0){0}
\put( 53,3){\line(1,0){10}}
\put( 65,0){0} \end{picture} & 9 & \text{yes} &-\\
\tilde{A}_2 & \begin{picture}(80,5)
\put( 0,0){0}
\put( 8,3){\line(1,0){10}}
\put( 20,0){0}
\put( 28,5){\line(1,0){15}}
\put( 28,2){\line(1,0){15}}
\put( 31,1){\mbox{$>$}}
\put( 45,0){0}
\put( 53,3){\line(1,0){10}}
\put( 65,0){2} \end{picture} & 9 & \text{yes} &-\\
A_2{+}\tilde{A}_1 & \begin{picture}(80,5)
\put( 0,0){0}
\put( 8,3){\line(1,0){10}}
\put( 20,0){0}
\put( 28,5){\line(1,0){15}}
\put( 28,2){\line(1,0){15}}
\put( 31,1){\mbox{$>$}}
\put( 45,0){1}
\put( 53,3){\line(1,0){10}}
\put( 65,0){0} \end{picture} & 7 & \text{no} & p\neq 2\\
B_2 & \begin{picture}(80,5)
\put( 0,0){2}
\put( 8,3){\line(1,0){10}}
\put( 20,0){0}
\put( 28,5){\line(1,0){15}}
\put( 28,2){\line(1,0){15}}
\put( 31,1){\mbox{$>$}}
\put( 45,0){0}
\put( 53,3){\line(1,0){10}}
\put( 65,0){1} \end{picture} & 6 & \text{no} & p\neq 2\\
\tilde{A}_2{+}A_1 & \begin{picture}(80,5)
\put( 0,0){0}
\put( 8,3){\line(1,0){10}}
\put( 20,0){1}
\put( 28,5){\line(1,0){15}}
\put( 28,2){\line(1,0){15}}
\put( 31,1){\mbox{$>$}}
\put( 45,0){0}
\put( 53,3){\line(1,0){10}}
\put( 65,0){1} \end{picture} & 6 & \text{no} &p\neq 3\\
C_3(a_1) & \begin{picture}(80,5)
\put( 0,0){1}
\put( 8,3){\line(1,0){10}}
\put( 20,0){0}
\put( 28,5){\line(1,0){15}}
\put( 28,2){\line(1,0){15}}
\put( 31,1){\mbox{$>$}}
\put( 45,0){1}
\put( 53,3){\line(1,0){10}}
\put( 65,0){0} \end{picture} & 5 & \text{no} &p\neq 2\\
F_4(a_3) & \begin{picture}(80,5)
\put( 0,0){0}
\put( 8,3){\line(1,0){10}}
\put( 20,0){2}
\put( 28,5){\line(1,0){15}}
\put( 28,2){\line(1,0){15}}
\put( 31,1){\mbox{$>$}}
\put( 45,0){0}
\put( 53,3){\line(1,0){10}}
\put( 65,0){0} \end{picture} & 4 & \text{yes} &-\\
B_3 & \begin{picture}(80,5)
\put( 0,0){2}
\put( 8,3){\line(1,0){10}}
\put( 20,0){2}
\put( 28,5){\line(1,0){15}}
\put( 28,2){\line(1,0){15}}
\put( 31,1){\mbox{$>$}}
\put( 45,0){0}
\put( 53,3){\line(1,0){10}}
\put( 65,0){0} \end{picture} & 3 & \text{yes} &-\\
C_3 & \begin{picture}(80,5)
\put( 0,0){1}
\put( 8,3){\line(1,0){10}}
\put( 20,0){0}
\put( 28,5){\line(1,0){15}}
\put( 28,2){\line(1,0){15}}
\put( 31,1){\mbox{$>$}}
\put( 45,0){1}
\put( 53,3){\line(1,0){10}}
\put( 65,0){2} \end{picture} & 3 & \text{yes} &-\\
F_4(a_2) & \begin{picture}(80,5)
\put( 0,0){0}
\put( 8,3){\line(1,0){10}}
\put( 20,0){2}
\put( 28,5){\line(1,0){15}}
\put( 28,2){\line(1,0){15}}
\put( 31,1){\mbox{$>$}}
\put( 45,0){0}
\put( 53,3){\line(1,0){10}}
\put( 65,0){2} \end{picture} & 2 & \text{yes} &-\\
F_4(a_1) & \begin{picture}(80,5)
\put( 0,0){2}
\put( 8,3){\line(1,0){10}}
\put( 20,0){2}
\put( 28,5){\line(1,0){15}}
\put( 28,2){\line(1,0){15}}
\put( 31,1){\mbox{$>$}}
\put( 45,0){0}
\put( 53,3){\line(1,0){10}}
\put( 65,0){2} \end{picture} & 1 & \text{yes} &-\\
F_4 & \begin{picture}(80,5)
\put( 0,0){2}
\put( 8,3){\line(1,0){10}}
\put( 20,0){2}
\put( 28,5){\line(1,0){15}}
\put( 28,2){\line(1,0){15}}
\put( 31,1){\mbox{$>$}}
\put( 45,0){2}
\put( 53,3){\line(1,0){10}}
\put( 65,0){2} \end{picture} & 0 & \text{yes} &-\\
\hline \multicolumn{5}{l}{\text{(Notation from \cite[p.~401]{C2})}}
\end{array}$}
\end{center}
\end{table}

There are eight weighted Dynkin diagrams which satisfy $\fg_d(1)\neq \{0\}$. 
We now consider these eight cases in detail, where we just focus on the
validity of condition (K2) in Definition~\ref{def24}. (In particular,
the Frobenius map $F\colon G\rightarrow G$ will not play a role in this 
section.)

\begin{abs} \label{bf4a}
Let $d(\alpha_1)=1$, $d(\alpha_2)=d(\alpha_3)=d(\alpha_4)=0$. We
have $\cb_d=16$ and 
\begin{align*}
\fg_d(1) &=\langle e_{1000}, e_{1100}, e_{1110}, e_{1120}, e_{1111}, 
e_{1220}, e_{1121}, \\&\qquad\quad e_{1221}, e_{1122}, e_{1231}, 
e_{1222}, e_{1232}, e_{1242}, e_{1342}\rangle_k,\\
\fg_d(2) &=\langle e_{2342}\rangle_k,
\end{align*}
where, for example, $1342$ stands for the root $\alpha_1+3\alpha_2+
4\alpha_3+2\alpha_4$. Let $\lambda\colon \fg_d(2)\rightarrow k$ be any
linear map. As in Example~\ref{bd4}, we work out the corresponding Gram 
matrix $\cG_\lambda$, where we set $x_1:=\lambda(e_{2342})$. It is given by
\begin{center}
$\cG_\lambda=\pm x_1\cdot\mbox{antidiag}(1,-1,2,-1,-2,1,2,-2, -1,2,1,-2,
1,-1)$.
\end{center}
We have $\det(\cG_\lambda)=64x_1^{14}$. Hence, if $p=2$, then the radical 
of $\sigma_\lambda$ is not zero. If $p\neq 2$, then the radical is zero 
whenever $x_1\neq 0$.
\end{abs}

\begin{abs} \label{bf4b}
Let $d(\alpha_1)=d(\alpha_2)=d(\alpha_3)=0$, $d(\alpha_4)=1$. We
have $\cb_d=13$ and 
\begin{align*}
\fg_d(1) &=\langle e_{0001},e_{0011},e_{0111}, e_{1111},e_{0121},
e_{1121},e_{1221},e_{1231}\rangle_k,\\
\fg_d(2)&=\langle e_{0122},e_{1122},e_{1222},e_{1232},e_{1242},e_{1342},
e_{2342}\rangle_k,
\end{align*}
where we use the same notational conventions as above. Let $\lambda\colon 
\fg_d(2)\rightarrow k$ be any linear map. Let $x_1,\ldots,x_8$ be the 
values of $\lambda$ on the $8$ basis vectors of $\fg_d(2)$ (ordered as 
above). Then we obtain
\begin{center} {\footnotesize 
$\cG_\lambda=\pm\left(\begin{array}{r@{\hspace{5pt}}r@{\hspace{5pt}}
r@{\hspace{5pt}}r@{\hspace{5pt}}r@{\hspace{5pt}}r@{\hspace{5pt}}
r@{\hspace{5pt}}r}
   0& 0& 0& 0& -2x_1& -2x_2& -2x_3& -x_4 \\
   0& 0& 2x_1& 2x_2& 0& 0& -x_4& -2x_5 \\
   0& -2x_1& 0& 2x_3& 0& x_4& 0& -2x_6 \\
   0& -2x_2& -2x_3& 0& -x_4& 0& 0& -2x_7 \\
   2x_1& 0& 0& x_4& 0& 2x_5& 2x_6& 0 \\
   2x_2& 0& -x_4& 0& -2x_5& 0& 2x_7& 0\\ 
   2x_3& x_4& 0& 0& -2x_6& -2x_7& 0& 0 \\ 
   x_4& 2x_5& 2x_6& 2x_7& 0& 0& 0& 0 
\end{array}\right).$}
\end{center}
In principle, we could work out $\det(\cG_\lambda)$ and then try to find 
out for which values of $x_1,\ldots,x_8$ it is non-zero. However, this 
determinant is already quite complicated; it is a linear combination of 
$34$ monomials in $x_1,\ldots,x_8$. But we can just notice that, if we 
set $x_4:=1$ and $x_i:=0$ for all $i\neq 4$, then $\det(\cG_\lambda)=1$.
So the radical of $\sigma_\lambda$ will be zero for this choice of 
$\lambda$, and this works for any field~$k$.
\end{abs}



\begin{abs} \label{bf4c}
Let $d(\alpha_1)=d(\alpha_3)=d(\alpha_4)=0$, $d(\alpha_2)=1$. We
have $\cb_d=10$ and  
\begin{align*}
\fg_d(1) &=\langle e_{0100},e_{1100},e_{0110},e_{1110},e_{0120},e_{0111},
\\&\qquad\quad e_{1120},e_{1111},e_{0121},e_{1121},e_{0122},
e_{1122}\rangle_k,\\ \fg_d(2)&=\langle e_{1220},e_{1221},e_{1231},
e_{1222},e_{1232},e_{1242} \rangle_k.
\end{align*}
Let $\lambda\colon \fg_d(2)\rightarrow k$ be any linear map, and denote by 
$x_1,\ldots,x_6$ the values of $\lambda$ on the $6$ basis vectors of 
$\fg_d(2)$ (ordered as above). Then we obtain
\begin{center} {\footnotesize 
$\cG_\lambda=\pm\left(\begin{array}{r@{\hspace{5pt}}r@{\hspace{5pt}}
r@{\hspace{5pt}}r@{\hspace{5pt}}r@{\hspace{5pt}}r@{\hspace{5pt}}
r@{\hspace{5pt}}r@{\hspace{5pt}}r@{\hspace{5pt}}r@{\hspace{5pt}}
r@{\hspace{5pt}}r}
   0& 0& 0& 0& 0& 0& -x_1& 0& 0& -x_2& 0& -x_4 \\
   0& 0& 0& 0& x_1& 0& 0& 0& x_2& 0& x_4& 0 \\
   0& 0& 0& 2x_1& 0& 0& 0& x_2& 0& -x_3& 0& -x_5 \\
   0& 0& -2x_1& 0& 0& -x_2& 0& 0& x_3& 0& x_5& 0 \\
   0& -x_1& 0& 0& 0& 0& 0& x_3& 0& 0& 0& -x_6 \\
   0& 0& 0& x_2& 0& 0& x_3& 2x_4& 0& x_5& 0& 0 \\
   x_1& 0& 0& 0& 0& -x_3& 0& 0& 0& 0& x_6& 0 \\
   0& 0& -x_2& 0& -x_3& -2x_4& 0& 0& -x_5& 0& 0& 0  \\
   0& -x_2& 0& -x_3& 0& 0& 0& x_5& 0& 2x_6& 0& 0  \\
   x_2& 0& x_3& 0& 0& -x_5& 0& 0& -2x_6& 0& 0& 0  \\
   0& -x_4& 0& -x_5& 0& 0& -x_6& 0& 0& 0& 0& 0  \\
   x_4& 0& x_5& 0& x_6& 0& 0& 0& 0& 0& 0& 0 
\end{array}\right).$}
\end{center}
Here we notice that, if we set $x_3:=1$, $x_4:=1$ and $x_i:=0$ for 
$i\neq 3,4$, then $\det(\cG_\lambda)=1$. Hence, the radical of
$\sigma_\lambda$ is zero for this choice of $\lambda$, and this works 
for any field~$k$.
\end{abs}

\begin{abs} \label{bf4d}
Let $d(\alpha_1)=d(\alpha_2)=d(\alpha_4)=0$, $d(\alpha_3)=1$. We
have $\cb_d=7$ and 
\begin{align*}
\fg_d(1) &=\langle e_{0010}, e_{0110}, e_{0011}, e_{1110}, e_{0111}, 
e_{1111}\rangle_k,\\
\fg_d(1) &=\langle e_{0120}, e_{1120}, e_{0121}, e_{1220}, e_{1121}, 
e_{0122}, e_{1221}, e_{1122}, e_{1222}\rangle_k.
\end{align*}
Let $\lambda\colon \fg_d(2)\rightarrow k$ be any linear map. Let $x_1,
\ldots,x_9$ be the values of $\lambda$ on the $9$ basis vectors of 
$\fg_d(2)$ (ordered as above). Then we obtain
\begin{center} {\footnotesize 
$\cG_\lambda=\pm\left(\begin{array}{r@{\hspace{5pt}}r@{\hspace{5pt}}
r@{\hspace{5pt}}r@{\hspace{5pt}}r@{\hspace{5pt}}r@{\hspace{5pt}}
r@{\hspace{5pt}}r@{\hspace{5pt}}r}
 0& 2x_1& 0& 2x_2& x_3& x_5 \\  -2x_1& 0& -x_3& 2x_4& 0& x_7 \\ 
   0& x_3& 0& x_5& 2x_6& 2x_8 \\ -2x_2& -2x_4& -x_5& 0& -x_7& 0 \\ 
   -x_3& 0& -2x_6& x_7& 0& 2x_9 \\  -x_5& -x_7& -2x_8& 0& -2x_9& 0 
\end{array}\right).$}
\end{center}
We have $\det(\cG_\lambda)=
16(x_1x_5x_9-x_1x_7x_8-x_2x_3x_9+x_2x_6x_7+x_3x_4x_8-x_4x_5x_6)^2$.
So, if $p=2$, then $\det(\cG_\lambda)=0$ (for any $\lambda$). If $p\neq 2$, 
then we notice that $\det(\cG_\lambda)=16$ for $x_4:=1$, $x_5:=1$, 
$x_6:=1$ and $x_i:=0$ for $i\neq 4,5,6$. Hence, the radical of 
$\sigma_\lambda$ is zero for this choice of~$\lambda$. 
\end{abs}

\begin{abs} \label{bf4e}
Let $d(\alpha_1)=2$, $d(\alpha_2)=d(\alpha_3)=0$, $d(\alpha_4)=1$. 
We have $\cb_d=6$ and 
\begin{align*}
\fg_d(1) &=\langle e_{0001}, e_{0011}, e_{0111}, e_{0121}\rangle_k,\\
\fg_d(2) &=\langle e_{1000}, e_{1100}, e_{1110}, e_{1120}, e_{1220}, 
e_{0122}\rangle_k.
\end{align*}
Let $\lambda\colon \fg_d(2)\rightarrow k$ be any linear map. Let $x_1,
\ldots,x_6$ be the values of $\lambda$ on the $6$ basis vectors of 
$\fg_d(2)$ (ordered as above). Then we obtain 
\begin{center}
$\cG_\lambda=\pm 2x_6\cdot \mbox{antidiag}(-1,1,-1,1)$.
\end{center}
We have $\det(\cG_\lambda)=16x_6^4$. Hence, if $p=2$, then the radical 
of $\sigma_\lambda$ is not zero. If $p\neq 2$, then the radical is zero 
whenever $x_6\neq 0$.
\end{abs}

\begin{abs} \label{bf4f}
Let $d(\alpha_1)=0$, $d(\alpha_2)=1$, $d(\alpha_3)=0$, $d(\alpha_4)=1$. 
We have $\cb_d=6$ and 
\begin{align*}
\fg_d(1) &=\langle e_{0100}, e_{0001}, e_{1100}, e_{0110}, e_{0011}, 
e_{1110}, e_{0120}, e_{1120}\rangle_k,\\
\fg_d(2) &=\langle e_{0111}, e_{1111}, e_{0121}, e_{1220}, e_{1121}
\rangle_k.
\end{align*}
Let $\lambda\colon \fg_d(2)\rightarrow k$ be any linear map. Let $x_1,
\ldots,x_5$ be the values of $\lambda$ on the $5$ basis vectors of 
$\fg_d(2)$ (ordered as above). Then we obtain
\begin{center} {\footnotesize 
$\cG_\lambda=\pm \left(\begin{array}{r@{\hspace{5pt}}r@{\hspace{5pt}}
r@{\hspace{5pt}}r@{\hspace{5pt}}r@{\hspace{5pt}}r@{\hspace{5pt}}
r@{\hspace{5pt}}r}
   0& 0& 0& 0& -x_1& 0& 0& -x_4 \\  0& 0& 0& -x_1& 0& -x_2& -x_3& -x_5 \\ 
   0& 0& 0& 0& -x_2& 0& x_4& 0 \\ 0& x_1& 0& 0& -x_3& 2x_4& 0& 0 \\ 
   x_1& 0& x_2& x_3& 0& x_5& 0& 0 \\ 0& x_2& 0& -2x_4& -x_5& 0& 0& 0 \\
   0& x_3& -x_4& 0& 0& 0& 0& 0 \\ x_4& x_5& 0& 0& 0& 0& 0& 0 
\end{array}\right).$}
\end{center}
We have $\det(\cG_\lambda)=9(x_1x_4^2x_5-x_2x_3x_4^2)^2$. Hence, if 
$p=3$, then the radical of $\sigma_\lambda$ is not zero. Now assume that 
$p\neq 3$. Then we notice that $\det(\cG_\lambda)=9$ for $x_1:=1$, 
$x_4:=1$, $x_5:=1$ and $x_i:=0$ for $i=2,3$. So the radical is zero for 
this choice of~$\lambda$.
\end{abs}

\begin{abs} \label{bf4g}
Let $d(\alpha_1)=1$, $d(\alpha_2)=0$, $d(\alpha_3)=1$, $d(\alpha_4)=0$. 
We have $\cb_d=5$ and 
\begin{align*}
\fg_d(1) &=\langle e_{1000}, e_{0010}, e_{1100}, e_{0110}, e_{0011}, 
e_{0111}\rangle_k,\\
\fg_d(2) &=\langle e_{1110}, e_{0120}, e_{1111}, e_{0121}, e_{0122}\rangle_k.
\end{align*}
Let $\lambda\colon \fg_d(2)\rightarrow k$ be any linear map. Let $x_1,
\ldots,x_5$ be the values of $\lambda$ on the $5$ basis vectors of 
$\fg_d(2)$ (ordered as above).  Then we obtain 
\begin{center} {\footnotesize 
$\cG_\lambda=\pm \left(\begin{array}{r@{\hspace{5pt}}r@{\hspace{5pt}}
r@{\hspace{5pt}}r@{\hspace{5pt}}r@{\hspace{5pt}}r}
0& 0& 0& x_1& 0& x_3 \\ 0& 0& x_1& 2x_2& 0& x_4 \\ 0& -x_1& 0& 0& -x_3& 0 \\
   -x_1& -2x_2& 0& 0& -x_4& 0 \\ 0& 0& x_3& x_4& 0& 2x_5 \\ 
   -x_3& -x_4& 0& 0& -2x_5& 0 \end{array}\right).$}
\end{center}
We have $\det(\cG_\lambda)=4(x_1^2x_5-x_1x_3x_4+x_2x_3^2)^2$. Hence, if 
$p=2$, then the radical of $\sigma_\lambda$ is not zero. Now assume that 
$p\neq 2$. Then we notice that $\det(\cG_\lambda)=4$ for $x_1:=1$, $x_5:=1$ 
and $x_i:=0$ for $i=2,3,4$. So the radical of $\sigma_\lambda$ is zero for 
this choice of~$\lambda$.
\end{abs}

\begin{abs} \label{bf4h}
Let $d(\alpha_1)=1$, $d(\alpha_2)=0$, $d(\alpha_3)=1$, $d(\alpha_4)=2$. 
We have $\cb_d=3$ and 
\begin{align*}
\fg_d(1) &=\langle e_{1000},e_{0010},e_{1100},e_{0110}\rangle_k,\\
\fg_d(2) &=\langle e_{0001},e_{1110},e_{0120}\rangle_k.
\end{align*}
Let $\lambda\colon \fg_d(2)\rightarrow k$ be any linear map. Let $x_1,
x_2,x_3$ be the values of $\lambda$ on the~$3$ basis vectors of 
$\fg_d(2)$ (ordered as above).  Then we obtain 
\begin{center} {\footnotesize 
$\cG_\lambda=\pm \left(\begin{array}{r@{\hspace{5pt}}r@{\hspace{5pt}}
r@{\hspace{5pt}}r} 0& 0& 0& x_2 \\ 0& 0& x_2& 2x_3 \\ 0& -x_2& 0& 0 
\\ -x_2& -2x_3& 0& 0 \end{array}\right).$}
\end{center}
Hence, we see that the radical of $\sigma_\lambda$ is zero for any 
linear map $\lambda\colon \fg_d(2)\rightarrow k$ such that 
$\lambda(e_{1110})=x_2 \neq 0$, and this works for any field~$k$.
\end{abs}

Similar computations can, of course, be performed for other types of 
groups. The results are summarized as follows.

\begin{abs} \label{corf4} Assume that the characteristic of $k$ is a
bad prime for $G$. Let $d\in \Delta$ be such that $\fg_d(1) \neq \{0\}$. 
Consider the following two statements.
\begin{itemize}
\item[(a)] If $d\in \Delta_k^{\!\bullet}$, then there exist linear maps 
$\lambda \colon \fg_d(2) \rightarrow k$ such that $\det(\cG_\lambda)
\neq 0$ and $\lambda(e_\alpha)\in \{0,1\}$ for all $\alpha\in \Phi_2$. 
\item[(b)] If $d\not\in \Delta_k^{\!\bullet}$, then $\det(\cG_\lambda)=0$
for all linear maps $\lambda \colon \fg_d(2)\rightarrow k$.
\end{itemize}
Let $G$ be simple of exceptional type $G_2$, $F_4$, $E_6$, $E_7$ or $E_8$. 
Then (a) can be verified by exactly the same kind of computations as in 
Example~\ref{be8}(a), by systematically running through all possibilities 
where $\lambda(e_\alpha)\in\{0,1\}$ for $\alpha\in \Phi_2$, until we find
one such that $\det(\cG_\lambda)\neq 0$. Even in type $E_8$, this just 
works in a few seconds. The verification of (b) is much harder. As in
the verification of Example~\ref{be8}(b), we need to show that the 
determinant of a certain matrix with entries in a polynomial ring over
$\F_p$ is~$0$. Except for some cases in type $E_8$, it is sufficient to 
use the Pfaffian of that matrix, as in Example~\ref{be8}(b). For types
$G_2$, $F_4$, we can see all this immediately from the results of the 
computations in Example~\ref{bg2} and in \ref{bf4a}--\ref{bf4h}, by 
comparing with the entries in the last column of Table~\ref{tab1}. 
However, there are critical cases in type $E_8$ (for example, 
$A_2{+}3A_1$, $2A_2{+}A_1$, $2A_2{+}2A_1$, $A_3{+}2A_1$, $A_3{+}A_2
{+}A_1$) where the computation of the Pfaffian appears to be practically
impossible. In these cases, some more sophisticated computational 
methods are required.
\end{abs}

\begin{prop}[Steel--Thiel \protect{\cite{ST}}] \label{corf4aa} Let
$G$ be of type $E_8$. Then the statement in \ref{corf4}(b) holds for
all $d\not\in\Delta\setminus \Delta_k^{\!\bullet}$. 
\end{prop}

(The proof relies on Groebner basis techniques.)

\begin{cor} \label{corf4a} If $G$ is simple of exceptional type $G_2$, 
$F_4$, $E_6$, $E_7$ or $E_8$, then Conjectures~\ref{main1a} and 
\ref{main2} hold for $G$.
\end{cor}

\begin{proof} First consider Conjecture~\ref{main1a}. Let $d\in 
\Delta_k^{\!\bullet}$. If $\fg_d(1)\neq \{0\}$, then we must show that 
there exists some $\lambda\colon \fg_d(2) \rightarrow k$ such that 
$\det(\cG_\lambda)\neq 0$. If $p$ (the characteristic of~$k$) is a good 
prime for $G$, then this holds by Remarks~\ref{rem24a} and~\ref{rem24c}($*$). 
If $p$ is a bad prime, then this holds since \ref{corf4}(a) is known to 
hold. Conversely, assume that either $\fg_d(1)=\{0\}$, or there exists a 
linear map $\lambda \colon \fg_d(2)\rightarrow k$ such that 
$\det(\cG_\lambda) \neq 0$. If $\fg_d(1)=\{0\}$, then $d\in
\Delta_{\text{spec}}\subseteq \Delta_k^{\!\bullet}$, as already remarked 
at the end of Section~\ref{sec3}. If $\fg_d(1)\neq \{0\}$, then we have 
$d\in \Delta_k^{\!\bullet}$ by \ref{corf4}(b) and Proposition~\ref{corf4aa}. 

Now consider Conjecture~\ref{main2}. Let $d\in \Delta_{\text{spec}}$.
If $\fg_{\Z,d}(1)\neq \{0\}$, then we must show that there exists a
homomorphism $\lambda\colon \fg_{\Z,d}(2) \rightarrow \Z$ such that 
$\sigma_\lambda\colon \fg_{\Z,d}(1)\times \fg_{\Z,d}(1)\rightarrow \Z$ 
is non-degenerate over $\Z$. The verification is similar to that in
\ref{corf4}(a), but now we work over $\Z$. Again, we 
systematically run through all possibilities where $\lambda(e_\alpha) 
\in\{0,1\}$ for $\alpha\in\Phi_2$, until we find one such that the Gram 
matrix of $\sigma_\lambda$ has determinant equal to~$1$. Even in type 
$E_8$, this just works in a few seconds. Conversely, assume
that either $\fg_{\Z,d}(1)=\{0\}$, or there exists a homomorphism $\lambda
\colon \fg_{\Z,d}(2) \rightarrow \Z$ such that $\sigma_\lambda
\colon \fg_{\Z,d}(1)\times \fg_{\Z,d}(1)\rightarrow \Z$ is 
non-degenerate over $\Z$. If $\fg_{\Z,d}(1)=\{0\}$, then $d\in
\Delta_{\text{spec}}$ (see again the remark at the end of
Section~\ref{sec3}). Now assume that $\fg_{\Z,d}(1)\neq \{0\}$.
By reduction modulo~$p$, we obtain a linear map $\lambda_k\colon
\fg_d(2)\rightarrow k$ and a corresponding alternating form 
$\sigma_{\lambda_k}\colon \fg_d(1)\times \fg_d(1)\rightarrow k$. Since
$\sigma_\lambda$ is non-degenerate over~$\Z$, the radical of 
$\sigma_{\lambda_k}$ will be zero and so $d\in \Delta_k^{\!\bullet}$, 
since we already know that Conjecture~\ref{main1a} is true for~$G$. 
Note that this holds for all choices of~$k$. So Table~\ref{tabb} shows 
that $d\in \Delta_{\text{spec}}$. 
\end{proof}


\begin{rem} \label{allan} Assume that $G$ is simple of type $A_n$.
Then there are no bad primes for $G$, and all unipotent classes of $G$ 
are special. Now Conjecture~\ref{main1a} is known to hold
in this case; see Remarks~\ref{rem24a} and~\ref{rem24c}($*$). As
far as Conjecture~\ref{main2} is concerned, it remains to show that
if $d\in \Delta$ is such that $\fg_{\Z,d}(1)\neq \{0\}$, then there
exists a homomorphism $\lambda\colon \fg_{\Z,d}(2) \rightarrow \Z$ 
such that $\sigma_\lambda \colon \fg_{\Z,d}(1)\times \fg_{\Z,d}(1)
\rightarrow \Z$ is non-degenerate over $\Z$. At the moment, we do not see
a general argument but, by similar methods as above, we have at least 
checked this holds for $2\leq n\leq 15$. 
\end{rem}

In order to verify Conjecture~\ref{main} in full, one would also need a 
description of the sets $\fg_d(2)^{*!}$; this will be discussed 
elsewhere. (For $G$ of type $A_n$, $B_n$, $C_n$, $D_n$, such a description 
is available from \cite{L3d}, \cite{xue1}.) Furthermore, one would need 
to check whether or not there exists some $\lambda$ which is defined over 
$\F_q$ and is in sufficiently general position.~---~It would be highly 
desirable to find a more conceptual explanation for all this.


\medskip
\noindent
{\bf Acknowledgements.}  I thank George Lusztig for discussions. Thanks are
also due for Gunter Malle for a careful reading of the manuscript and many
useful comments. I am indebted to Ulrich Thiel and Allen Steel for 
addressing, and solving, the computational issues in type $E_8$ (see 
\cite{ST}), which provided crucial support for Conjectures~\ref{main1a} 
and~\ref{main2}. Thanks are also due to Alexander Premet for comments and
for pointing out the reference \cite{prem2}. Most of the work was done while
I enjoyed the hospitality of the University of Cyprus at Nicosia (early 
October 2018); I thank Christos Pallikaros for the invitation. This work 
is a contribution to DFG SFB-TRR 195 ``Symbolic tools in Mathematics and 
their application''.



\begin{thebibliography}{131}



\bibitem{C1} 
{\sc R. W.~Carter}, Centralizers of semisimple elements in the finite
classical groups, Proc. London Math. Soc. {\bf 42} (1981), 1--41.

\bibitem{C2} 
{\sc R. W.~Carter}, {\it Finite groups of Lie type: Conjugacy classes 
and complex characters}, Wiley, New York, 1985.


\bibitem{clpr}
{\sc M. C. Clarke and A. Premet}, The Hesselink stratification of nullcones
and base change, Invent. Math. {\bf 191} (2013), 631--669.

 
\bibitem{dress}
{\sc A. W. M. Dress and W. Wenzel}, A simple proof of an identity
concerning Pfaffians of skew symmetric matrices, Advances in Math. 
{\bf 112} (1995), 120--134.

\bibitem{duma}
{\sc O. Dudas and G. Malle}, Modular irreducibility of cuspidal 
unipotent characters, Invent. Math. {\bf 211} (2018), 579--589.

\bibitem{Eno1}
{\sc H. Enomoto}, The characters of the finite symplectic group 
$\mbox{Sp}(4,q)$, $q=2^f$, Osaka J. Math. {\bf 9} (1972), 75--94.

\bibitem{Eno2}
{\sc H. Enomoto}, The characters of the finite Chevalley group $G_2(q)$,
  $q=3^f$, Japan.\ J.~Math. {\bf 2} (1976), 191--248.

\bibitem{EnYa}
{\sc H.~Enomoto and H.~Yamada}, The characters of $G_2(2^n)$, Japan.\
J.~Math. {\bf 12} (1986), 325--377.

\bibitem{gap4}
The GAP Group, \textit{GAP -- Groups, Algorithms, and Programming}.
  Version 4.8.10, 2018. (\url{http://www.gap-system.org})

\bibitem{myg}
{\sc M. Geck}, On the construction of semisimple Lie algebras and Chevalley
groups, Proc. Amer. Math. Soc.  {\bf 145} (2017), 3233--3247.


\bibitem{gehe}
{\sc M.~Geck and D.~H\'ezard}, On the unipotent support of character
sheaves, Osaka J. Math. {\bf 45} (2008), 819--831.

\bibitem{gehi2}
{\sc M.~Geck and G.~Hiss}, Modular representations of finite groups of Lie
type in non-defining characteristic, {\it in: Finite reductive groups}
(Luminy, 1994; ed. M.~Cabanes), Progress in Math. {\bf 141}, pp.~195--249,
Birkh\"auser, Boston, MA, 1997.

\bibitem{gema}
{\sc M. Geck and G.~Malle}, On the existence of a unipotent support for the
irreducible characters  of finite groups of Lie type,  Trans. Amer. Math.
Soc. {\bf 352} (2000), 429--456.



\bibitem{kaw0} 
{\sc N. Kawanaka}, Generalized Gelfand-Graev representations and
Ennola duality, {\it in: Algebraic Groups and Related Topics}, 
Advanced Studies in Pure Math. {\bf 6}, Kinokuniya, Tokyo, 
and North-Holland, Amsterdam, 1985, 175--206.

\bibitem{kaw1}
{\sc N. Kawanaka}, Generalized Gelfand-Graev representations of
exceptional algebraic groups I, Invent.\ Math. {\bf 84} (1986), 575--616.

\bibitem{kaw2} 
{\sc N. Kawanaka}, Shintani lifting and Gelfand-Graev representations, 
Proc.\ Symp.\ Pure Math., Amer.\ Math.\ Soc. {\bf 47} (1987), 147--163.

\bibitem{knu}
{\sc D. E. Knuth}, Overlapping pfaffians, Electron. J. Combin. {\bf 3} (1996),
R5.

\bibitem{Lue07}
{\sc F. L\"ubeck}, Centralizers and numbers of semisimple classes in 
exceptional groups of Lie type, see
  \url{http://www.math.rwth-aachen.de/~Frank.Luebeck/chev/CentSSClasses}.

\bibitem{Ls1}
{\sc G.~Lusztig}, A class of irreducible representations of a Weyl group, 
Proc. Kon. Nederl. Akad. (A) {\bf 82} (1979), 323--335.

\bibitem{L1} 
{\sc G.~Lusztig}, {\it Characters of reductive groups over a finite field},
Ann.\ Math.\ Studies {\bf 107}, Princeton U.\ Press, 1984.

\bibitem{L2}
{\sc G.~Lusztig}, A unipotent support for irreducible representations, 
Adv. Math. {\bf 94} (1992), 139--179. 

\bibitem{Ls2}
{\sc G.~Lusztig}, Notes on unipotent classes, Asian J. Math. {\bf 1} (1997), 
194--207.

\bibitem{L3a}
{\sc G. Lusztig}, Unipotent elements in small characteristic, Transform.
Groups {\bf 10} (2005), 449--487. 

\bibitem{L3b}
{\sc G. Lusztig}, Unipotent elements in small characteristic II, Transform.
Groups {\bf 13} (2008), 773--797. 

\bibitem{L4}
{\sc G. Lusztig}, Unipotent classes and special Weyl group representations, 
J. Algebra {\bf 321} (2009), 3418--3449. 

\bibitem{L3c}
{\sc G. Lusztig}, Unipotent elements in small characteristic III, 
J. Algebra {\bf 329} (2011), 163--189. 

\bibitem{L3d}
{\sc G. Lusztig}, Unipotent elements in small characteristic IV, Transform.
Groups {\bf 15} (2010), 921--936. 

\bibitem{LuSp}
{\sc G.~Lusztig and N.~Spaltenstein}, Induced unipotent classes, J. London 
Math.\ Soc. {\bf 19} (1979), 41--52.


\bibitem{prem}
{\sc A. Premet}, Nilpotent orbits in good characteristic and the 
Kempf--Rousseau theory. Special issue celebrating the 80th birthday 
of Robert Steinberg. J. Algebra {\bf 260} (2003), 338--366.

\bibitem{prem2}
{\sc A. Premet}, A modular analogue of Morozov's theorem on maximal 
subalgebras of simple Lie algebras, Advances in Math. {\bf 311} (2017),
833--884.

\bibitem{spa}
{\sc N. Spaltenstein}, {\it Classes unipotentes et sous-groupes de Borel},
Lecture Notes in Math. 946, Springer, Berlin Heidelberg New York,  1982.

\bibitem{spr}
{\sc T. A. Springer}, {\it Linear algebraic groups, second edition},
Birkh\"auser, Boston, 1998.

\bibitem{ST}
{\sc A. Steel and U. Thiel}, Special unipotent classes and weighted 
Dynkin diagrams: computations in type $E_8$, {\it in preparation}.

\bibitem{St}
{\sc R.~Steinberg}, {\it Lectures on Chevalley groups}, mimeographed notes,
Department of Math., Yale University, 1967/68; now available as vol.~66 of
the University Lecture Series, Amer. Math. Soc., Providence, R.I., 2016.

\bibitem{tay}
{\sc J. Taylor}, Generalized Gelfand--Graev representations in small 
characteristics, Nagoya Math. J. {\bf 224} (2016), 93--167. 


\bibitem{xue1}
{\sc T. Xue}, Nilpotent elements in the dual of odd orthogonal algebras,
Transform.  Groups {\bf 17} (2012), 571--592. 



\end{thebibliography}
\end{document}